\documentclass[a4paper]{article}
\usepackage{geometry,amsmath,amssymb,graphicx,enumerate,bbm,theorem}
\usepackage[british]{babel}

\newcommand{\Iota}{\mathrm{I}}
\newcommand{\Mho}{{\mbox{\rotatebox[origin=c]{180}{$\Omega$}}}}
\newcommand{\assign}{:=}
\newcommand{\comma}{{,}}
\newcommand{\mathi}{\mathrm{i}}
\newcommand{\nin}{\not\in}
\newcommand{\nocomma}{}
\newcommand{\tmdummy}{$\mbox{}$}
\newcommand{\tmem}[1]{{\em #1\/}}
\newcommand{\tmmathbf}[1]{\ensuremath{\boldsymbol{#1}}}
\newcommand{\tmop}[1]{\ensuremath{\operatorname{#1}}}
\newcommand{\tmstrong}[1]{\textbf{#1}}
\newcommand{\tmtextit}[1]{{\itshape{#1}}}
\newenvironment{enumeratealpha}{\begin{enumerate}[a{\textup{)}}] }{\end{enumerate}}
\newenvironment{enumeratenumeric}{\begin{enumerate}[1.] }{\end{enumerate}}
\newenvironment{itemizedot}{\begin{itemize} }{\end{itemize}}
\newenvironment{proof}{\noindent\textbf{Proof\ }}{\hspace*{\fill}$\Box$\medskip}
\newtheorem{corollary}{Corollary}
\newtheorem{lemma}{Lemma}
\newtheorem{proposition}{Proposition}
{\theorembodyfont{\rmfamily}\newtheorem{remark}{Remark}}
\newtheorem{theorem}{Theorem}

\title{Integrability of potentials of degree $k \neq \pm 2$. Second order
variational equations between Kolchin solvability and
Abelianity\\[1em]
{\normalsize{Guillaume Duval and Andrzej J.~ Maciejewski}}}

\begin{document}

\maketitle

\begin{abstract}
  In our previous paper: {\tmem{Integrability of Homogeneous potentials of
  degree $k = \pm 2$. An application of higher order variational equations}},
  we tried to extract some particular structures of the higher variational
  equations (the $\tmop{VE}_p$ for $p \geqslant 2$), along particular
  solutions of some Hamiltonian systems. Then, we use them to get new Galois
  obstructions to the integrability of natural Hamiltonian with potential of
  degree $k = \pm 2$. In the present work, we apply the results of the
  previous paper, to the complementary cases, when the degrees of the
  potentials are relative integers $k$, with $|k| \geqslant 3$. Since these
  cases are much more general and complicated, we reduce our study only to the
  second variational equation $\tmop{VE}_2$.
\end{abstract}

\section{Presentation}

This paper completes the previous one {\cite{duval12}} by applying the same
results and strategies to the complementary cases. That is, we still study the
integrability of homogeneous potentials along Darboux points, but here we
assume that the degree $k$ is an arbitrary relative integer with $|k|
\geqslant 3$. As we shall see thing are technically much more complicated for
the following reasons. At first, the assumption over the degree are more
general, and secondly, here in contrast to the cases when $|k| = 2$, the
Morales Ramis table (see Table \ref{tableMR}) gives discreet obstructions at
the level of the first order variational equation $\tmop{VE}_1$. This will
force the study to encompass a judge number of distinct cases. As a
consequence, our major results, Proposition \ref{prop:test1}, Theorems
\ref{prop:EXtest2} and \ref{prop:ve2test2} below, just concern the Galois
group of the second variational equation $\tmop{VE}_2$. They are not
definitive results which guarantee that the Galois groups of the associated
systems $\tmop{VE}^{\gamma}_{2, \alpha}$ and $_{} \tmop{EX}_{2, \alpha,
\beta}^{\gamma}$ are virtually Abelian in such and such cases, but they
constitute effective algorithms to test such possibilities in particular
cases. These results convert the virtual Abelianity of the Galois group to
testing some {\tmem{Ostrowski relations}} between first level integrals
$\Phi$, $\Psi_{\alpha}$, $\Psi_{\gamma}$ etc, when they are so. The reader
will see that in general, these integrals are very complicated. And the main,
difficulty will be be the following : from the form of a given integrals, it
will be in general very easy to predict that it can be algebraic if it is so.
But in contrast when it does not have such specific form, it is very very
difficult to decide when it is transcendental. Our main ingredient for this
will be Remark \ref{rem:fuchs} below.

In addition to the first paper, this one contains one additional idea which is
interesting from a theoretical point of view. This is what we call the
{\tmem{cohomological argument}}, which allow to test that a second level
integral is indeed computable in closed form, that is to test if a Galois
group is virtually Abelian, without computing effectively this second level
integral in closed form explicitly. This procedure was discovered when dealing
with the present second level integrals, but it can be applied in more general
contexts.

Let us mention finally, that this approach is not isolated. Combot in
{\cite{combot12}}, deal with the same problem for homogeneous potentials of
degree $k = - 1$, while in the same time Weil in {\cite{weil12}}, is studying
exactly the same problem than us but from the complementary point of view of
{\tmem{gauge-transformations}}, which has the advantage to convert, the
original systems into new ones which are much more simple in the sens that the
virtual Abelianity of the Galois group can be directly read in the form of the
new system.

In order to helps the reader, into the reading of this paper, let's briefly
present how it is organised. In Section 2, we present the systems
$\tmop{VE}^{\gamma}_{2, \alpha}$ and $_{} \tmop{EX}_{2, \alpha,
\beta}^{\gamma}$. Here the idea is that as for the study of $\tmop{VE}_1$,
these new systems are much more easy to study after the so called
{\tmem{Yoshida-transformation}}, than in their original time-parametrisation.
In Section 3, we present the associated second level integrals involved in
those systems and their intrinsic hierarchy. Section 4, contains the technical
ingredients both theoretical and practical which are going to be useful for
the proofs of Theorems \ref{prop:EXtest2} and \ref{prop:ve2test2} (in Section
5), and for some effective remarks about $\tmop{VE}^{\gamma}_{2, \alpha}$ and
$\tmop{EX}_{2, \alpha, \beta}^{\gamma}$. These practical studies will be done
in Sections 6 and 7 below. Moreover, we recommend the reader, to first take a
look, to Section \ref{sec:exp}, where we present some experimental facts about
the complexity of the law that govern the virtual Abelianity of
$\tmop{VE}^{\gamma}_{2, \alpha}$. We hope that this will help him to
understand how we are using the present criteria and why we where obliged to
deal with such tremendous casuistic.

\section{From $\tmop{VE}_1$ to $\tmop{VE}_2$ through the Yoshida
transformations}

In all the previous works concerning the integrability of homogeneous
potentials along {\tmem{Darboux points}}, the key result for the analysis of
the first variational equation $\tmop{VE}_1$, was \ the conversion of this
differential system in time parametrisation \ to an \ equivalent one in new
variable $z$. This is the Yoshida transformation given by
\[ t \longmapsto z = \varphi^k (t) . \]
This transformation convert the initial variational equations over a
hyperelliptic curve into a Fuchsian equation over $\mathbbm{P}^1$, with
singularities at $z \in \{0 ; 1 ; \infty\}$.

In this Section we shall recall some results concerning this transformation,
and we show how it applies to the study of the second variational equation
$\tmop{VE}_2$.

\subsection{The subsystems of $\tmop{VE}_2$ to deal with}

Here, we assume that $V'' (c)$ is diagonalisable. Hence, $\tmop{VE}_1$ splits
into a direct sum of equations which have the following form
\[ \ddot{x} = - \lambda_{\alpha} \varphi^{k - 2} x \nocomma, \]
where $\lambda_{\alpha}$ are the eigenvalues of $V'' (c)$, and $\varphi =
\varphi ( t)$ is a particular solution defined by a Darboux point
$\tmmathbf{d}$. For each of these equations we denote by $G_1 = G_{\alpha}
=\mathcal{G} (k, \lambda_{\alpha})$ its differential Galois group over the
field $K =\mathbbm{C} (\varphi, \dot{\varphi})$.

From Proposition 2.5 of {\cite{duval12}}, we know that the differential
Galois group of $\tmop{VE}_2$ is virtually Abelian iff \ the same property
hold true for the differential Galois groups of the systems
\[ \tmop{VE}_{2, \alpha}^{\gamma}  \hspace{1em} \tmop{and} \hspace{1em}
   \tmop{EX}_{2, \alpha, \beta}^{\gamma} \text{\tmop{for} \tmop{all}}
   \hspace{1em} \alpha ; \beta ; \gamma \hspace{1em} \tmop{with} \hspace{1em}
   \alpha \neq \beta . \]
In time parametrisation, these systems have the following form (see equations
(2.20) and (2.22) of {\cite{duval12}})
\[ \tmop{VE}_{2, \alpha}^{\gamma}  \hspace{1em} \left\{ \begin{array}{lll}
     \ddot{x} & = & - \lambda_{\alpha} \varphi^{k - 2} x\\
     \ddot{y} & = & - \lambda_{\gamma} \varphi^{k - 2} y + \varphi^{k - 3} x^2
   \end{array} \right. ; .. \]
and
\[ \tmop{EX}_{2, \alpha, \beta}^{\gamma} \hspace{1em} \left\{
   \begin{array}{lll}
     \ddot{x} & = & - \lambda_{\alpha} \varphi^{k - 2} x\\
     \ddot{y} & = & - \lambda_{\beta} \varphi^{k - 2} y\\
     \ddot{u} & = & - \lambda_{\gamma} \varphi^{k - 2} u + \varphi^{k - 3} xy
   \end{array} \right. \]

Here, we directly assumed that the coefficients appearing in the
non-homogeneous terms in the right hand sides of the last equations of those
systems are non zero. Otherwise, the corresponding second variational equation
will be reduced to the first variational \ equation.

\subsection{The Morales-Ramis table}

The full list of all values of $\lambda \in \mathbbm{C}$ for which the
differential Galois group of equation
\begin{equation}
  \ddot{x} = - \lambda_{} \varphi^{k - 2} x, \label{eqVElambda}
\end{equation}
is virtually Abelian is given in the following table, \ where $p$ denotes an \
integer.

\begin{table}[h]
  \begin{tabular}{|c|c|c|c|}
    \hline
    $G (k, \lambda)^{\circ}$ & $k$ & $\lambda$ & Line Number\\
    \hline
    & $k = \pm 2$ & $\lambda$ arbitrary complex number & 1\\
    \hline
    $G_a$ & $|k| \geqslant 3$ & $\lambda (k ; p) = p + \frac{k}{2} p (p - 1)$
    & 2\\
    \hline
    & 1 & $p + \frac{1}{2} p (p - 1)$, $p \neq - 1 ; 0$ & 3\\
    \hline
    & -1 & $p - \frac{1}{2} p (p - 1)$, $p \neq 1 ; 2$ & 4\\
    \hline
    $\{\tmop{Id}\}$ & 1 & $0$ & 5\\
    \hline
    & -1 & $1$ & 6\\
    \hline
    & $|k| \geqslant 3$ & $\frac{1}{2}  \left( \frac{k - 1}{k} + p (p + 1) k)
    \right.$ & 7\\
    \hline
    & 3 & $\frac{- 1}{24} + \frac{1}{6}  (1 + 3 p)^2$, $\frac{- 1}{24} +
    \frac{3}{32}  (1 + 4 p)^2$ & 8,9\\
    \hline
    &  & $\frac{- 1}{24} + \frac{3}{50}  (1 + 5 p)^2$, $\frac{- 1}{24} +
    \frac{3}{50}  (2 + 5 p)^2$ & 10,11\\
    \hline
    & -3 & $\frac{25}{24} - \frac{1}{6}  (1 + 3 p)^2$, $\frac{25}{24} -
    \frac{3}{32}  (1 + 4 p)^2$ & 12,13\\
    \hline
    &  & $\frac{25}{24} - \frac{3}{50}  (1 + 5 p)^2$, $\frac{25}{24} -
    \frac{3}{50}  (2 + 5 p)^2$ & 14,15\\
    \hline
    & 4 & $\frac{- 1}{8} + \frac{2}{9}  (1 + 3 p)^2$ & 16\\
    \hline
    & -4 & $\frac{9}{8} - \frac{2}{9}  (1 + 3 p)^2$ & 17\\
    \hline
    & 5 & $\frac{- 9}{40} + \frac{5}{18}  (1 + 3 p)^2$, $\frac{- 9}{40} +
    \frac{1}{10}  (2 + 5 p)^2$ & 18,19\\
    \hline
    & -5 & $\frac{49}{40} - \frac{5}{18}  (1 + 3 p)^2$, $\frac{49}{40} -
    \frac{1}{10}  (2 + 5 p)^2$ & 20,21\\
    \hline
  \end{tabular}
  \caption{\label{tableMR}}
\end{table}

Since we will now play with potential of degree $k \neq \pm 2$, the connected
component of $G_1$ will be either $G_a$ or identity. See
{\cite{moralesramis01}} and {\cite{duval08}}.

\subsection{The Yoshida transformation of $\tmop{VE}_1$}

Here we reproduce some of the computations which were more detailled in
{\cite{duval08}}. The Yoshida transformation consists of the change of the
independent variable
\begin{equation}
  t \longmapsto z = \varphi^k (t), \label{eqyoshida} \nocomma
\end{equation}
in the considered equation. Thanks to the chain rule
\[ \frac{d^2 x}{dt^2} = \left( \frac{dz}{dt} \right)^2  \frac{d^2 x}{dz^2} +
   \frac{d^2 z}{dt^2}  \frac{dx}{dz}, \]
equation (\ref{eqVElambda}) reads
\begin{eqnarray}
  \varphi^{k - 2} (t)  [ 2 kz (1 - z)  \frac{d^2 x}{dz^2} + (2 (k - 1) (1 - z)
  - kz)  \frac{dx}{dz}] & = & \frac{d^2 x}{dt^2} \nonumber\\
  \varphi^{k - 2} (t)  [ 2 kz (1 - z)  \frac{d^2 x}{dz^2} + (2 (k - 1) (1 - z)
  - kz)  \frac{dx}{dz}] & = & - \varphi^{k - 2} (t) \lambda x \nonumber\\
  2 kz (1 - z)  \frac{d^2 x}{dz^2} + (2 (k - 1) (1 - z) - kz)  \frac{dx}{dz} &
  = & - \lambda x \nonumber\\
  \frac{d^2 x}{dz^2} + \frac{2 (k - 1) (1 - z) - kz}{2 kz (1 - z)} 
  \frac{dx}{dz} & = & \frac{- \lambda x}{2 kz (1 - z)} \nonumber\\
  \frac{d^2 x}{dz^2} + p (z)  \frac{dx}{dz} & = & s (z) \lambda x, 
\end{eqnarray}
\label{eqVEz}where
\[ p (z) = \frac{2 (k - 1) (z - 1) + kz}{2 kz (z - 1)} \text{ \hspace{1em}
   \tmop{and}}  \hspace{1em} s (z) = \frac{1}{2 kz (z - 1)} . \]
Now, after the classical Tschirnhaus change of dependent variable,
\begin{equation}
  x = f (z) \zeta, \hspace{2em} f (z) = \exp \left( \frac{1}{2} \int p (z) dz
  \right) = z^{\frac{- (k - 1)}{2 k}}  (z - 1)^{\frac{- 1}{4}},
\end{equation}
\label{eqtchirn}equation (\ref{eqVEz}) has the \ reduced form
\begin{equation}
  \frac{d^2 \zeta}{dz^2} = [ r_0 (z) + \lambda s (z) ] \zeta,
\end{equation}
\label{eqVEredz}where
\[ r (z) : = r_{\lambda} ( z) : = r_0 (z) + \lambda s (z) = \frac{\rho^2 -
   1}{4 z^2} + \frac{\sigma^2 - 1}{4 (z - 1)^2} - \frac{1}{4}  (1 - \rho^2 -
   \sigma^2 + \tau^2) \left( \frac{1}{z} + \frac{1}{1 - z} \right), \]
and
\begin{equation}
  \rho = \frac{1}{k}, \hspace{2em} \sigma = \frac{1}{2}, \hspace{2em} \tau =
  \frac{\sqrt{(k - 2)^2 + 8 k \lambda}}{2 k} \label{eqdeltaexp} .
\end{equation}
Since the three above numbers are respectively the exponents differences at $z
= 0$, \ $z = 1$, and $z = \infty$, of the reduced hypergeometric equation $L_2
= x'' - r (z) x = 0$; the solutions of $L_2 = 0$, belong to the Riemann scheme
\begin{equation}
  P \{ \begin{array}{ccc}
    0 & 1 & \infty\\
    \frac{1}{2} - \frac{1}{2 k} & \frac{1}{4} & \frac{- 1 - \tau}{2}\\
    \frac{1}{2} + \frac{1}{2 k} & \frac{3}{4} & \frac{- 1 + \tau}{2}
  \end{array} z \} \label{eqRS} .
\end{equation}
In Table \ref{tableMR}, the group $G (k, \lambda)$ appearing into the first
column is precisely the differential Galois group of the equation $L_2 = x'' -
r (z) x = 0$, with respect \ to the ground field $\mathbbm{C} (z)$.

\subsection{Computation of the solution space of (\ref{eqVElambda}) when $G
(k, \lambda)^{\circ} = G_a$.}

We observed that after Yoshida transformation the new equation in $z$ variable
is $L_2 = x'' - r_{\lambda} (z) x = 0$, and that \ the solutions of $L_2 = 0$
belong to the Riemann scheme
\begin{equation}
  P_1 \assign P \{ \begin{array}{ccc}
    0 & 1 & \infty\\
    \frac{1}{2} - \frac{1}{2 k} & \frac{1}{4} & \rho_{\infty}\\
    \frac{1}{2} + \frac{1}{2 k} & \frac{3}{4} & \rho'_{\infty}
  \end{array} z \} \label{eqRS},
\end{equation}
where
\begin{equation}
  \rho_{\infty} = \frac{- 1 - \tau}{2} \nocomma \comma \hspace{2em}
  \rho_{\infty}' = \frac{- 1 + \tau}{2} \nocomma \comma \hspace{2em} \tau = p
  - \frac{1}{2} + \frac{1}{k} \label{eqtau} .
\end{equation}
In Table \ref{tableMR}, the group $G (k, \lambda)$ appearing into the first
column is precisely the differential Galois group of the equation $L_2 = x'' -
r (z) x = 0$, with respect to the ground field $\mathbbm{C} (z)$. But its
connected component coincide with the connected component of
(\ref{eqVElambda}) over $\mathbbm{C} (\varphi, \dot{\varphi})$.

Our Lemma 3.4 from {\cite{duval08}} can be reformulated into the following
way.

\begin{lemma}
  \label{lem:jordan}When $G_1^{\circ} \simeq G_a$ then,
  \begin{enumeratenumeric}
    \item Up to a complex multiplicative constant, the algebraic solution
    $x_1$ is of the form $x_1 = z^a  (z - 1)^b J (z)$ where
    \[ a \in \left\{ \frac{k - 1}{2 k}, \frac{k + 1}{2 k} \right\} \nocomma,
       \hspace{2em} b \in \left\{ \frac{1}{4}, \frac{3}{4} \right\}, \]
    \ and $J (z) \in \mathbbm{R} [z]$ does not vanish at $z \in \{0, 1\}$.
    
    \item The function $I \assign I_{\lambda} = \int \frac{1}{x_1^2}$ has
    exponent $( 1 - 2 b)$ at $z = 1$, \ and, up to an additive constant, the
    monodromy around this \ point is \ $\mathcal{M}_1 (I) = - I$.
    
    \item Around $z = 0$, $I$ has exponent $( 1 - 2 a)$, and its monodromy can
    be written into the form $\mathcal{M}_0 (I) = \exp (- 4 \pi \mathi a) I +
    c_0$.
  \end{enumeratenumeric}
\end{lemma}

\begin{proof}
  (1 and 2) where proved in Lemma 3.4 of {\cite{duval08}}, (3) follows in the
  same way.
\end{proof}

From (\ref{eqtau}) and (\ref{eqRS}) we get
\[ \rho_{\infty} = - \frac{1}{2}  \left( p + \frac{1}{2} + \frac{1}{k}
   \right) ; \hspace{2em} \rho'_{\infty} = \frac{1}{2}  \left( - \frac{3}{2} +
   \frac{1}{k} \right) . \]
Since $x_1 = z^a  (z - 1)^b J (z)$, we get that
\[ x_1 \in P_1 \assign P \{ \begin{array}{ccc}
     0 & 1 & \infty\\
     \frac{1}{2} - \frac{1}{2 k} & \frac{1}{4} & \rho_{\infty}\\
     \frac{1}{2} + \frac{1}{2 k} & \frac{3}{4} & \rho'_{\infty}
   \end{array} z \} \hspace{1em} \Longleftrightarrow \hspace{1em} J \in P_2
   \assign P \{ \begin{array}{ccc}
     0 & 1 & \infty\\
     \frac{1}{2} - \frac{1}{2 k} - a & \frac{1}{4} - b & \rho_{\infty} + a +
     b\\
     \frac{1}{2} + \frac{1}{2 k} - a & \frac{3}{4} - b & \rho'_{\infty} + a +
     b
   \end{array} z \} . \]
From the four possibilities $a \in \{ \frac{k - 1}{2 k}, \frac{k + 1}{2 k}
\}$, $b \in \{ \frac{1}{4}, \frac{3}{4} \}$, we get four possibles Riemann
schemes.

According to ({\cite{poole}} p. 95), the classical Jacobi polynomials
${J^{\ast}}_n^{(\alpha, \beta)} (t) $with parameters $(\alpha, \beta)$, and
$\tmop{degree} n \in \mathbbm{N}$, are defined by
\[ {J^{\ast}}_n^{(\alpha, \beta)} (t) = \frac{(t - 1)^{- \alpha}  (t + 1)^{-
   \beta}}{2^n n!}  \frac{d^n}{dt^n} ( (t - 1)^{\alpha + n} (t + 1)^{\beta +
   n}) . \]
They belong to the Riemann scheme
\[ P_{J^{\ast}} \left\{ \begin{array}{ccc}
     - 1 & \infty & 1\\
     0 & - n & 0\\
     - \beta & \alpha + \beta + n + 1 & - \alpha
   \end{array} t \right\} . \]
But here the singularities that we meet are $\{- 1, 1, \infty\}$ instead of
$\{0, 1, \infty\}$. Hence we pass from the classical Jacobi polynomials to
ours by putting $t = 2 z - 1 \nocomma \nocomma \nocomma$, i.e., we set
\[ J (z) = J^{\ast} (2 z - 1), \]
\[ J (z) \in P_{J^{}} \left\{ \begin{array}{ccc}
     0 & \infty & 1\\
     0 & - n & 0\\
     - \beta & \alpha + \beta + n + 1 & - \alpha
   \end{array} z \left\} \hspace{1em} \Longleftrightarrow \hspace{1em}
   J^{\ast} (t) \in P_{J^{\ast}} \left\{ \begin{array}{ccc}
     - 1 & \infty & 1\\
     0 & - n & 0\\
     - \beta & \alpha + \beta + n + 1 & - \alpha
   \end{array} t \right\} . \right. \right. \]
As a consequence, we get the following classification of the Jacobi
polynomials $J \in P_2$

\begin{table}[h]
  \begin{tabular}{|c|c|c|c|c|c|c|}
    \hline
    & $a$ & $b$ & $\alpha$ & $\beta$ & $n$ & $p$\\
    \hline
    Case 1 & $\frac{1}{2} + \frac{1}{2 k}$ & $\frac{1}{4}$ & $\frac{- 1}{2}$ &
    $\frac{1}{k}$ & $\frac{p - 1}{2}$ & $p \in 2\mathbbm{N}+ 1$\\
    \hline
    Case 2 & $\frac{1}{2} + \frac{1}{2 k}$ & $\frac{3}{4}$ & $\frac{1}{2}$ &
    $\frac{1}{k}$ & $\frac{p}{2} - 1$ & $p \in 2\mathbbm{N}+ 2$\\
    \hline
    Case 3 & $\frac{1}{2} - \frac{1}{2 k}$ & $\frac{1}{4}$ & $\frac{- 1}{2}$ &
    $\frac{- 1}{k}$ & $\frac{- p}{2}$ & $p \in - 2\mathbbm{N}$\\
    \hline
    Case 4 & $\frac{1}{2} - \frac{1}{2 k}$ & $\frac{3}{4}$ & $\frac{1}{2}$ &
    $\frac{- 1}{k}$ & $\frac{- p - 1}{2}$ & $p \in - 2\mathbbm{N}- 1$\\
    \hline
  \end{tabular}
  \caption{\label{tabjacobi}(}
\end{table}

Here, $n \assign \deg (J (z))$. To obtain this, we identified $P_2$ with
$P_J$. This gave $\alpha$ and $\beta$. To compute the degree, we observed that
$\alpha + \beta \nin \mathbbm{Z}$. Hence, $- n$ is the one of the two numbers
\ $\rho_{\infty} + a + b$, or \ $\rho'_{\infty} + a + b$ that belongs to
$\mathbbm{Z}$.

Thanks to the formula for $J^{\ast} = {J^{\ast}_n}^{(\alpha, \beta)}$ we get
that up to a constant multiple
\begin{equation}
  J (z) = J_n^{(\alpha, \beta)} (z) = (z - 1)^{- \alpha}  (z)^{- \beta} 
  \frac{d^n}{dz^n} ( (z - 1)^{\alpha + n} (z)^{\beta + n}) . \label{eqjacobi}
\end{equation}
Hence, according to the previous table and Lemma \ref{lem:jordan}, the precise
forms of the integrals $I$ depends on the four \ cases and are given by the
following

\begin{table}[h]
  \begin{tabular}{|c|c|c|c|c|}
    \hline
    Cases & 1 & 2 & 3 & 4\\
    \hline
    $I'$ & $\frac{1}{z^{1 + 1 / k}  (z - 1)^{1 / 2} J^2}$ & $\frac{1}{z^{1 + 1
    / k}  (z - 1)^{3 / 2} J^2}$ & $\frac{1}{z^{1 - 1 / k}  (z - 1)^{1 / 2}
    J^2}$ & $\frac{1}{z^{1 - 1 / k}  (z - 1)^{3 / 2} J^2}$\\
    \hline
  \end{tabular}
  \caption{\label{tab:I}}
\end{table}

\begin{remark}
  Compatibility between the above Table and Lemma is certainly true for $|k|
  \geqslant 3$ but there is a problem for $k = \pm 1$. Indeed, thanks to
  ({\cite{poole}} p.95), the differential equation for $J (z) =
  J_n^{(\alpha, \beta)} (z)$ is
  \[ z (1 - z) J'' + [\beta + 1 - (\alpha + \beta + 2) z] J' + n (\alpha +
     \beta + n + 1) J = 0. \]
  So by plugging $z = 0, 1$ in the later we get
  \[ \left\{ \begin{array}{l}
       (\beta + 1) J' (0) + n (\alpha + \beta + n + 1) J (0) = 0\\
       - (\alpha + 1) J' (1) + n (\alpha + \beta + n + 1) J (1) = 0
     \end{array} \right. . \]
  From Table \ref{tabjacobi}, $\alpha = \pm 1 / 2$ and $\beta = \pm 1 / k$ so
  $n (\alpha + \beta + n + 1) \neq 0$. As a consequence, $J (1) \neq 0$. But
  when $\beta = - 1$ that is when $k = \pm 1$, we get that $J (0) = 0$. Which
  is not compatible with the Lemma \ref{lem:jordan}.
\end{remark}

This is why, in this paper, we are going to work with the assumption that $|k|
\geqslant 3$, which is in fact complementary to what was done in
{\cite{combot12}}.

\subsection{Yoshida transformation of $\tmop{VE}_{2, \alpha}^{\gamma}$ and
$\tmop{EX}_{2, \alpha, \beta}^{\gamma}$.}

With the above notations we get the following

\begin{proposition}
  By setting $x = f (z) \zeta$, the differential equation in time $t$:
  \[ \ddot{x} = - \lambda \varphi^{k - 2} x + \varphi^{k - 3} b, \]
  is transformed into the differential equation in $z$ variable
  \[ \zeta'' = r_{\lambda} (z) \zeta - \frac{s (z)}{f (z) z^{1 / k}} b \]
\end{proposition}

\begin{proof}
  According to Yoshida, by dividing the original equation in time by
  $\varphi^{k - 2}$ and multiplying by $- s (z)$ we get
  \[ x'' + p (z) x' = \lambda s (z) x - \frac{s (z) b}{\varphi} . \]
  Now, the relation $x = f (z) \zeta$, gives $x'' + p (z) x' = f (z) \zeta'' +
  (f'' + pf') \zeta$, hence by dividing the previous equation by $f$ and using
  {\cite{duval08}} again, we get
  \begin{eqnarray*}
    \zeta'' - r_0 (z) \zeta & = & \frac{\lambda s (z)}{f} f \zeta - \frac{s
    (z) b}{f (z) \varphi}\\
    \zeta'' & = & r_{\lambda} (z) \zeta - \frac{s (z) b}{f (z) z^{1 / k}}
  \end{eqnarray*}
\end{proof}

Let us set
\[ \omega = z^{- ( \frac{3}{2} + \frac{1}{2 k})}  (z - 1)^{\frac{- 5}{4}} .
\]
The field $\mathbbm{C} (z) [\omega]$ is a finite Abelian extension of
$\mathbbm{C} (z)$.

\begin{corollary}
  \label{cor:ve2z}The original systems in time $\tmop{VE}_{2, \alpha}^{\gamma}
  $ and $\tmop{EX}_{2, \alpha, \beta}^{\gamma}$ are equivalent to the
  following systems in z variable that we still denote by the same symbols:
  \[ \tmop{VE}_{2, \alpha}^{\gamma} \hspace{1em} \left\{ \begin{array}{l}
       x'' = r_{\alpha} x,\\
       y'' = r_{\gamma} y + \omega x^2,
     \end{array} \right. \]
  
  \[ \tmop{EX}_{2, \alpha, \beta}^{\gamma} \hspace{1em} \left\{
     \begin{array}{l}
       x'' = r_{\alpha} x,\\
       y'' = r_{\beta} y,\\
       u'' = r_{\gamma} u + \omega xy,
     \end{array} \right. \]
  The coefficients of these systems are elements of $K_0 \assign \mathbbm{C}
  (z) [\omega]$. Here, in order to simplify notations we set:
  \[ r_{\alpha} = r_{\lambda_{\alpha}} (z) = r_0 (z) + \lambda_{\alpha} s (z)
     . \]
\end{corollary}

\begin{proof}
  In $\tmop{EX}_{2, \alpha, \beta}^{\gamma}$, let's perform the Tschirnhaus
  transformation by setting
  \[ x (t) \assign f (z) X (z) ; \hspace{2em} y (t) : = f (z) Y (z) ;
     \hspace{2em} u (t) \assign f (z) U (z) . \]
  According to the above proposition, the system in time is therefore
  equivalent to
  \[ \tmop{EX}_{2, \alpha, \beta}^{\gamma} \hspace{1em} \left\{
     \begin{array}{l}
       X'' = r_{\alpha} X\\
       X'' = r_{\beta} Y\\
       U'' = r_{\gamma} U - \frac{s (z) }{f (z) z^{1 / k}} xy
     \end{array} \right. . \]
  But now,
  \[ \frac{s (z) }{f (z) z^{1 / k}} xy = \frac{s (z) }{f (z) z^{1 / k}} f^2
     XY = \frac{s (z) }{z^{1 / k}} f (z) XY. \]
  Moreover direct computation gives
  \[ \frac{s (z) }{z^{1 / k}} f (z) = \frac{\omega}{2 k} . \]
  Hence, the last equation of $\tmop{EX}_{2, \alpha, \beta}^{\gamma}$ becomes
  \[ U'' = r_{\gamma} U - \frac{\omega}{2 k} XY. \]
  Coming back to minuscule letter by setting
  \[ X = - 2 kx ; \hspace{2em} Y = y ; \hspace{2em} U = u, \]
  we get the desired expression of $\tmop{EX}_{2, \alpha, \beta}^{\gamma}$
  into $z$ variable. Similar computations hold for $\tmop{VE}^{\gamma}_{2,
  \alpha}$.
\end{proof}

\section{Second level integrals and hierarchy}

\subsection{The second level integral involved in $\tmop{VE}_{2,
\alpha}^{\gamma} $ and $\tmop{EX}_{2, \alpha, \beta}^{\gamma}$}

From now on, we will only work with differential equation in $z$ variable,
with coefficient in $K_0 =\mathbbm{C} (z) [\omega]$.

\begin{proposition}
  \label{prop:int}Let $F$ be a differential field containing the elements $r$,
  $\omega$, $b$, and a basis of solutions $\{u_1 ; u_2 \}$ of $u'' = ru$.
  Then, the field $F$ contains all solutions of the differential equation
  \[ u'' = ru + \omega b, \]
  iff it contains \ the two following integrals
  \[ \Phi_i \assign \int \omega u_i b \nocomma \nocomma, \hspace{1em} i = 1,
     \nocomma 2. \]
\end{proposition}

\begin{proof}
  Classical variation of constant.
\end{proof}

Let us \ denote by $F_1 / K_0$ the \ Picard-Vessiot extension associated to
the homogeneous part of one of the systems $S \assign \tmop{VE}_{2,
\alpha}^{\gamma}$, or $S \assign \tmop{EX}_{2, \alpha, \beta}^{\gamma}$. Let
also denote by $F_2 / K_0$ the Picard-Vessiot of $S$. From the above
\begin{itemizedot}
  \item For $S = \tmop{VE}_{2, \alpha}^{\gamma}$, $F_2 / F_1$ is generated by
  the four integrals $\Phi = \int \omega yx^2$. That is, in term of basis
  $\{x_1 ; x_2 \}$ and $\{y_1 ; y_2 \}$ of the corresponding homogeneous
  equations, we have at most six generators, given by the integrals
  \[ \Phi : = \int \omega y_i X, \hspace{1em} i = 1 ; 2. \]
  where $X \in \{x_1^2 \text{,} x_1 x_2 \text{,} x_2^2 \} .$
  
  \item For $S = \tmop{EX}_{2, \alpha, \beta}^{\gamma}$, $F_2 / F_1$ is
  generated by at most eight integrals $\Phi_{i, j, l} \assign \int \omega u_i
  y_j x_l$, with $i, j, l = 1, 2.$
\end{itemizedot}
These systems are complicated ones, nevertheless we get the following
elimination result

\begin{theorem}
  \label{th:elimination} If the groups $G_{\alpha}, G_{\beta}, G_{\gamma}$ are
  finite then the Picard-Vessiot extensions of $\tmop{VE}_{2,
  \alpha}^{\gamma}$, and $\tmop{EX}_{2, \alpha, \beta}^{\gamma}$ are virtually
  Abelian.
\end{theorem}

\begin{proof}
  In these cases $F_1 / K_0$ is algebraic. It is also the algebraic closure of
  $K$ in $F_2$. Since $F_2 / F_1$ is generated by first level integral with
  respect \ to $F_1$, $\tmop{Gal}^{\circ} ( F_2 / K_0)$ is a vector group.
\end{proof}

{\tmstrong{From now, since the casuistic is sufficiently tremendous, with
those cases only, we shall assume up to the end of this paper that }}
\[ G_{\alpha}^{\circ} \simeq G_{\beta}^{\circ} \simeq G_{\gamma}^{\circ}
   \simeq G_a . \]
\subsection{Hierarchy of the integrals involved in $\tmop{VE}_{2,
\alpha}^{\gamma} $ and \ $\tmop{EX}_{2, \alpha,
\beta}^{\gamma}$.}\label{sec:hierarchy}

Since in the considered cases the group $G_{\lambda}$ of $w'' = r_{\lambda} w$
has the connected component $G_a$, the equation always have an algebraic
solution that we shall always denote by an index one, i.e., $w_1$. The second
solution is given by
\[ w_2 = w_1 I_{\lambda} \text{\tmop{with}} I = I_{\lambda} = \int
   \frac{1}{w_1^2} . \]
With this convention the second level integral involved in $\tmop{VE}_{2,
\alpha}^{\gamma}$ can be classify in following diagram
\[ \begin{array}{cc}
     \Phi = \int \omega y_1 x^2_1, & \\
     \Phi_{\gamma} = \int \omega y_2 x^2_1 = \int \Phi' I_{\gamma}, &
     \Phi_{\alpha} = \int \omega y_1 x_1 x_2 = \int \Phi' I_{\alpha},\\
     \Phi_{\gamma, \alpha} = \int \omega y_2 x_1 x_2 = \int \Phi' I_{\gamma}
     I_{\alpha}, & \Phi_{2 \alpha} = \int \omega y_1 x^2_2 = \int \Phi'
     I^2_{\alpha},\\
     & \Phi_{\gamma, 2 \alpha} = \int \omega y_2 x^2_2 = \int \Phi'
     I_{\gamma} I^2_{\alpha}
   \end{array} \]
Observe that $\Phi$ is a first level integral with respect to \ the algebraic
extension $K$ of $\mathbbm{C} (z)$, with
\[ K \assign \mathbbm{C} ( z) [ \omega ; y_1 ; x_1] = K_0 [ y_1 ; x_1] . \]
The remaining five integrals, are of second level with respect to $K$. Their
complexity grow at each change of line in the diagram.

For $\tmop{EX}_{2, \alpha, \beta}^{\gamma}$, we proceed similarly. Here, $\Phi
= \int \omega u_1 y_1 x_1$ is first level with respect to
\[ K \assign \mathbbm{C} (z) [\omega ; u_1 ; y_1 ; x_1] = K_0 [ u_1 ; y_1 ;
   x_1] . \]
The diagram of complexity is now the following
\[ \begin{array}{ccc}
     & \Phi = \int \omega u_1 y_1 x_1, & \\
     \Phi_{\gamma} = \int \Phi' I_{\gamma}, & \Phi_{\beta} = \int \Phi'
     I_{\beta}, & \Phi_{\alpha} = \int \Phi' I_{\alpha},\\
     \Phi_{\gamma, \beta} = \int \Phi' I_{\gamma} I_{\beta}, & \Phi_{\gamma,
     \alpha} = \int \Phi' I_{\gamma} I_{\alpha}, & \Phi_{\beta, \alpha} = \int
     \Phi' I_{\beta} I_{\alpha},\\
     & \Phi_{\gamma, \beta, \alpha} = \int \Phi' I_{\gamma} I_{\beta}
     I_{\alpha}, & 
   \end{array} \]

\section{Ingredients and tools}

\subsection{A cohomological argument to decide between solvability and
Abelianity}

Here we recall one important result stated and proved in the first part of
this paper (Theorem 3.1 of {\cite{duval12}}). And we add to it some
refinements which are going to be useful for effective testing of the virtual
Abelianity.

\begin{theorem}
  \label{thliouville} Let $F_1 / K$ and $F_2 / K$ be two Picard-Vessiot
  extensions with $F_1 \subset F_2$ and $F_2 / F_1$ generated by integral of
  second level. Then $G_2$ is virtually Abelian iff $G_1$ is virtually
  Abelian, and any second level integrals $\Phi \in F_2$ can be expanded into
  the form
  \[ \Phi = R_1 + J, \]
  where $R_1 \in T (F_1 / K)$ and $J' \in \tilde{K}$. Here $\tilde{K}$ is the
  algebraic closure of $K$ in $F_2$. Moreover, for all $\sigma \in
  G_2^{\circ}$, $\sigma (\Phi) - \Phi \in T (F_1 / K)$.
\end{theorem}

Here we recall that $T (F_1 / K) $ stands for the Picard-Vessiot ring of the
extension $F_1 / K$.

Let $K \subset F_1 \subset F_2$ be a tower a Picard-Vessiot extensions
satisfying the assumptions of Theorem \ref{thliouville}, with $F_1 / K$
virtually Abelian. Let $\Phi \in F_2$ be a second level integral. If $G_2$ is
virtually Abelian, the mapping
\[ \sigma \longmapsto \mathcal{C} (\sigma) \assign \sigma (\Phi) - \Phi, \]
is a {\tmem{cocycle}} from $G_2^{\circ}$ with values in the
$G_2^{\circ}$-module $M = T (F_1 / K)$. Indeed, $\mathcal{C}$ is a cocycle
since it satisfies the relation
\[ \mathcal{C} (s \sigma) = s \cdot \mathcal{C} (\sigma) +\mathcal{C} (s) . \]
Now, let $R \in M = T (F_1 / K)$, and $h$ be an arbitrary mapping from
$G_2^{\circ}$ to $\tmop{the} \tmop{field} \tmop{of} \tmop{constats} C$, we say
that the mapping
\[ \mathcal{B}: G_2^{\circ} \rightarrow M ; \hspace{2em} \sigma \mapsto
   \sigma (R) - R + h (\sigma), \]
is an {\tmem{extended coboundary}} from $G_2^{\circ}$ with values in $M$. When
moreover, $h \in \tmop{Hom} (G_2^{\circ}, C)$, then we say that $\mathcal{B}$
is a {\tmem{coboundary}}. Direct computation shows that an extended coboundary
$\mathcal{B}$ is a cocycle iff $\mathcal{B}$ is a coboundary.

In this language, Theorem \ref{thliouville}, says that if $G_2^{\circ}$ is
Abelian, then any cocycle associated to a second level integral is a
coboundary. We may wonder about a converse. Let $\Phi$ be a second level
integral such that there exists $R \in T (F_1 / K),$ and some function $h :
G_2^{\circ} \rightarrow C$, (here we do not assume a priory that $h$ is
additive), satisfying
\[ \sigma (\Phi) - \Phi = \sigma (R) - R + h (\sigma), \hspace{2em} \forall
   \sigma \in G_2^{\circ} . \]
This relation yield the following implications
\begin{eqnarray*}
  \sigma (\Phi - R) & = & \Phi - R + h (\sigma)\\
  \sigma (\Phi' - R') & = & \Phi' - R'\\
  \Phi' - R' & \in & \bar{K} .
\end{eqnarray*}
Hence, there exists $J$ such that $J'$ is algebraic over $K$, and \ $\Phi = R
+ J$. As a consequence, we proved that the cocycle associated to a second
level integral $\Phi$ is a coboundary iff $\Phi$ can be computed in closed
form.

Let us conclude this sub-section by explaining how we will use these
observations. \ For any given tower of Picard-Vessiot extensions and an
explicate second level integral $\Phi$, the fact that the associated cocycle
$\mathcal{C} (\sigma) = \sigma (\Phi) - \Phi$ belong to $T (F_1 / K)$, will
give us necessary conditions for the virtual Abelianity of $G_2$. Next, thanks
to the previous arguments, we will find sufficient conditions, showing that
the cocycles are coboundaries.

\subsection{Ostrowski relations and necessary conditions for virtual
Abelianity}

The following lemma was also stated into the first part of this paper as a
consequence of Theorem \ref{thliouville}, (see Lemma 3.3 of {\cite{duval12}}).
But we state it again in a more appropriated version to the present context.
Moreover, we prove it again since its proof is better understandable into the
framework of the cohomological arguments.

\begin{lemma}
  \label{lem:fund}Let $K \subset F_1 \subset F_2$ be a tower of Picard-Vessiot
  extensions of $K$, with the same field of constants $C$, and satisfying the
  following conditions
  \begin{itemizedot}
    \item $G_1^{\circ} = G_a$, $T (F_1 / K) = K [I]$, \ and $I' \in K$,
    
    \item $F_2$ contains some second level integrals of the form $\Phi_1
    \assign \int \Phi' I$, with $\Phi' \in K,$ and $\Phi \in F_2$.
  \end{itemizedot}
  If $F_2 / K$ is virtually Abelian, then $\Phi$ and $I$ satisfies Ostrowski
  relation of the form
  \[ \Phi - dI \in K, \text{\tmop{for} \tmop{some}} d \in C. \]
  Conversely, if such a relation holds then the extension $K (I, \Phi_1) / K$
  is virtually Abelian.
\end{lemma}

\begin{proof}
  There exists $c \in \tmop{Hom} (G_2^{\circ}, \mathbbm{C})$ such that for all
  $\sigma \in G_2^{\circ}$, we have $\sigma (I) = I + c (\sigma)$. Hence,
  \[ \sigma (\Phi_1') = \Phi_1' + c (\sigma) \Phi' \Rightarrow \sigma (\Phi_1)
     - \Phi_1 = c (\sigma) \Phi + d (\sigma), \]
  for some mapping $d : G_2^{\circ} \rightarrow \mathbbm{C}$. Now according to
  Theorem \ref{thliouville}, the virtual Abelianity of $F_2 / K$ implies that
  the cocycles $\sigma (\Phi_1) - \Phi_1 \in T (F_1 / K) = K [I]$. Hence
  \[ \sigma (\Phi_1) - \Phi_1 = c (\sigma) \Phi + d (\sigma) \in T (F_1 / K)
     = K [I] . \]
  Let us choose $\sigma = \sigma_0$ such that $c (\sigma_0) = 1$. The last
  relation implies that $\Phi \in K [I]$. Hence, the two primitive integrals
  over $K$, $\Phi$ and $I$ must be dependant and we conclude thanks to
  Ostrowski.
  
  For the converse, let us \ assume that $\Phi = dI + f$ for some $f \in K$.
  Integrating by part, we can compute $\Phi_1$ in closed form thanks to first
  level integrals. Indeed we get
  \[ \Phi_1 = dI^2 / 2 + fI - \int fI' . \]
  and the claim follows since, $\int fI'$ is a first level integral w.r.t $K$.
  
  Alternatively, let us show on this example how the previous cohomological
  arguments are working here. If we assume that $\Phi = dI + f$ for some $f
  \in K$, we get a closed expression for the cocycle
  \[ \sigma (\Phi_1) - \Phi_1 = \Phi c (\sigma) + d (\sigma) = dIc (\sigma) +
     fc (\sigma) + d (\sigma) \in T (F_1 / K) = K [I] . \]
  Let us show that this is an extended coboundary. If we set $P (I) \assign
  dI^2 / 2 + fI$, we get
  \begin{eqnarray*}
    \sigma (P (I)) - P (I) & = & P (I + c (\sigma)) - P (I)\\
    & = & \frac{d}{2}  (I + c (\sigma))^2 + f (I + c (\sigma)) - \frac{d}{2}
    I^2 - fI\\
    \sigma (P (I)) - P (I) & = & dIc (\sigma) + fc (\sigma) + \frac{dc^2
    (\sigma)}{2} .
  \end{eqnarray*}
  Hence, by comparing the above two formulae, we get
  \[ \sigma (\Phi_1) - \Phi_1 = \sigma (P (I)) - P (I) + h (\sigma), \]
  where the function $h : G_2^{\circ} \rightarrow C$, can be computed thanks
  to the formula given by \ $h (\sigma) = - dc^2 (\sigma) / 2 + d (\sigma)$.
  Since the cocycle: $\sigma (\Phi_1) - \Phi_1$ is an extended coboundary,
  and, in fact, is a coboundary, $\Phi_1$ can be computed in closed form, \
  and the result follows.
\end{proof}

This proof explains why we introduced, the a priory artificial notion of an
extended coboundary. Indeed, when looking to the complicated formula for $h$
above, it not obvious that it is a group morphism from $G_2^{\circ}$ to $C$.
Nevertheless, what is really important for our purpose is that $h$ takes
constant values. In the more complicated cases that we shall meet below, we
will not explicitly compute $h$, but \ we will only show its existence.
Moreover, the advantage of this cohomological approach is that it shows that a
second level integrals can be explicitly computed in closed form without
having to make this computation explicitly. This will make things simpler in
the more complicated cases below. \

\subsection{Testing Ostrowski relations thanks to characters}

Let $K /\mathbbm{C} (z)$ be an algebraic extension. \ Let $\Phi$ and $I$ be
two primitive integrals of elements belonging to $K$. In order to test if they
satisfy an Ostrowski relation of the form
\begin{equation}
  \Phi + dI = p \in K \label{eq:test}
\end{equation}
for some $d \in \mathbbm{C}$, we shall use the following observation by taking
advantage that $\Phi$ and $I$ are primitive of algebraic functions.

Let $\sigma$ be a Galois morphism fixing $\mathbbm{C} (z)$, for example a
monodromy operator, and assume further that $\sigma$ acts on $\Phi'$ and $I'$
by characters according to the formulae
\[ \sigma (\Phi') = \chi (\Phi) \Phi', \hspace{2em} \sigma (I') = \chi (I) I'
   . \]
By integrating we get relations of the form
\[ \sigma (\Phi) = \chi (\Phi) \Phi + c_{\Phi}, \hspace{2em} \sigma (I) =
   \chi (I) I + c_I . \]
As a consequence, applying $\sigma$ to (\ref{eq:test}), we \ get a system of
two equations
\[ \left(\begin{array}{cc}
     1 & 1\\
     \chi (\Phi) & \chi (I)
   \end{array}\right)  \left(\begin{array}{c}
     \Phi\\
     dI
   \end{array}\right) = \left(\begin{array}{c}
     p\\
     \sigma (p) - c_{\Phi} - dc_I
   \end{array}\right) \in K^2 . \]
If $\chi (\Phi) \neq \chi (I)$ the matrix is invertible and (\ref{eq:test})
implies that $\Phi$ is algebraic. We have therefore proved the following
criteria.

\begin{lemma}
  \label{lem:char}{\tmdummy}
  
  \begin{enumeratenumeric}
    \item Assume that $\Phi$ is not algebraic, \ and that there exists \ a
    monodromy operator $\sigma$ acting by character on $\Phi'$ and $I'$. If
    $\chi (\Phi) \neq \chi (I)$, then (\ref{eq:test}) does not hold.
    
    \item Here is a generalisation: Let $\Phi, I_1, \ldots, I_n$ be $n + 1$
    integrals over $K$. Let $\sigma$ be a Galois morphism acting by characters
    on the derivatives
    \[ \sigma (\Phi') = \chi_{\Phi} \Phi', \hspace{2em} \sigma (I_j') =
       \chi_j I'_j, \hspace{2em}  \text{\tmop{for}} j = 1, \nocomma \ldots, n.
    \]
    If, $\chi_{\Phi} \nin \{\chi_1, \ldots, \chi_n \}$, then an arbitrary
    Ostrowski relation of the form
    \[ \Phi + \sum^n_{i = 1} d_j^{} I_j \in K, \]
    implies that $\Phi \in K$
  \end{enumeratenumeric}
\end{lemma}

The proof of the second point \ is similar to previous particular case.
Indeed, by grouping together the $I_j$ corresponding to the same character, we
are reduced to the case where all the $\chi_j$ are distinct. Then, the action
of the $\sigma^p$ for $0 \leqslant p \leqslant n$ lead to an invertible
Vandermonde $n \times n$ system, which allows a similar conclusion as in point
(1).

We are going to use this lemma in the proofs of Propositions \ref{prop:R1ve2},
\ref{prop:R2ve2} and \ref{prop:R1EX} below, by showing thanks to monodromies
that some Ostrowski relations are impossible. This is one of the great
advantage of Yoshida transformation in comparison to the time parametrisation,
where the singularities of the corresponding complex functions are not well
understood. Another advantage of Yoshida transformation is going to be shown
right now.

\subsection{To be or not to be an algebraic integral}

\label{sec:tobe}

\begin{remark}
  \label{rem:fuchs} Let $F / K_0$ be the Picard-Vessiot extension of either
  $\tmop{VE}_{2, \alpha}^{\gamma}$, or $\tmop{EX}_{2, \alpha,
  \beta}^{\gamma}$, over the field $K_0 =\mathbbm{C} (z) [\omega]$. The two
  above systems are Fuchsian with singularities at $z \in \{0, 1, \infty\}$.
  Indeed, according to Corollary \ref{cor:ve2z} and Proposition
  \ref{prop:int}, each solution of any of these systems is holomorphic in any
  simply connected domain of $\mathbbm{P}^1 \backslash \{0 ; 1 ; \infty\}$,
  with at most exponential growth at the singularities.
  
  As a consequence, the Schlesinger \ theorem implies that $\tmop{Gal} (F
  /\mathbbm{C}(z))$ is topologically generated by the two monodromies
  $\mathcal{M}_0$ and $\mathcal{M}_1$.
  
  This observation will have the following important consequence. Let $\Gamma$
  be a holonomic element of $F$ fixed by $\mathcal{M}_0$, and having a finite
  orbit under $<\mathcal{M}_1 >$. More generally, let us assume that $\Gamma$
  has a finite orbit under the monodromy group $\mathcal{M} \subset \tmop{Gal}
  (F /\mathbbm{C}(z))$. Then necessarily, $\Gamma$ is algebraic over
  $\mathbbm{C} (z)$. Indeed, since $\Gamma$ is holonomic, there exists a
  $\mathbbm{C}$-finite dimensional vector space $V$ containing $\Gamma$ on
  which $\tmop{Gal} (F /\mathbbm{C}(z))$ acts algebraically. As a consequence,
  the map
  \[ f : \tmop{Gal} (F /\mathbbm{C}(z)) \rightarrow V ; \hspace{2em} \sigma
     \mapsto \sigma (\Gamma), \]
  is a morphism of algebraic variety. Since the image $f (\mathcal{M})$ is
  finite, it is a Zariski closed subset of $V$, so $f^{- 1} (f (\mathcal{M}))$
  is a closed subset of $\tmop{Gal} (F /\mathbbm{C}(z))$ containing
  $\mathcal{M}$. Since $\mathcal{M}$ is dense in $\tmop{Gal} (F
  /\mathbbm{C}(z))$, we get that $f^{- 1} (f (\mathcal{M})) = \tmop{Gal} (F
  /\mathbbm{C}(z))$ and the orbit of $\Gamma$ under $\tmop{Gal} (F
  /\mathbbm{C}(z))$ is finite as has to be shown.
\end{remark}

The most general integrals $\Phi$ and $\Psi$ we shall meet below are Abelian
integrals of the form
\[ \Gamma \assign \int \frac{P \Omega}{J^2}, \]
where
\begin{itemizedot}
  \item $\Omega (z) = z^{e_0}  (1 - z)^{e_1}$, where the exponents $e_0$ and
  $e_1$ are rational numbers $> - 1$, with $e_0 + e_1 \nin \mathbbm{Z}$.
  
  \item $P (z) \in \mathbbm{C} [z]$.
  
  \item $J (z)$ is a Jacobi polynomial having $n$ simples roots $0 < z_1 <
  \cdots < z_n < 1$, if $n = \deg (J) \geqslant 1$.
\end{itemizedot}
Precisely, $\Gamma$ will be an integral of the type $\Phi$, for $n = \deg (J)
= 0$, and of the type $\Psi$ otherwise.

According to Remark \ref{rem:fuchs}, if one of the two exponents $e_0$, or
$e_1$ is an integer, then the corresponding integral $\Phi,$ or $\Psi$ has
finite orbit under the monodromy group, hence is algebraic. This is a very
surprising fact especially for the integrals $\Psi$. Indeed, this shows that
here, the $P, \Omega$ and $J$ must be so specific that $P \Omega / J^2$ does
not has residues at none of the $z_i$.

Away from those cases, we now have to investigate the integrals $\Gamma$, for
which the two exponents are not integers in order to be able to test their
eventual algebraicity.

\subsubsection{Reduction of the integrals}

When trying to compute $\Gamma$ in closed form we get the following formula
\[ \forall R \in \mathbbm{C} (z), \hspace{2em} \left( \frac{R \Mho}{J}
   \right)' = \frac{T (R) \Omega}{J^2}, \]
with
\[ \left\{ \begin{array}{lll}
     \Mho (z) & \assign & \Omega (z) z (1 - z)\\
     T (R) & \assign & z (1 - z) JR' + [(e_0 + 1 - (e_0 + e_1 + 2) z) J + z (z
     - 1) J'] R
   \end{array} \right. . \]
Viewed as a linear mapping of $\mathbbm{C} (z)$ to itself, $T$ is injective.
Indeed,
\[ (T (R) = 0) \hspace{1em} \Longrightarrow \hspace{1em} \left( \frac{R
   \Mho}{J} \right)' = 0 \hspace{1em} \Longrightarrow \hspace{1em} (R = 0) .
\]
Now, if we restrict $T$ to $\mathbbm{C} [z]$, by computing the leading term of
$T (z^r)$, we see that the condition $e_0 + e_1 \nin \mathbbm{Z}$ implies that
$T$ increment the degree by $n + 1$. That is $\deg (T (R)) = \deg (R) + n +
1$. By counting dimensions we therefore get $\tmop{that} \tmop{for} \tmop{all}
N \geqslant 0$, we have the following direct sum decomposition
\[ \mathbbm{C}_{N + n + 1} [z] = T (\mathbbm{C}_N [z]) \oplus \mathbbm{C}_n
   [z] . \]
Since it holds for all $N \geqslant 0$, we get
\[ \mathbbm{C} [z] = T (\mathbbm{C}[z]) \oplus \mathbbm{C}_n [z] . \]
As a consequence, we can reduce any integral $\Gamma$ by lowering the degree
of the numerator in the following way : $\forall P \in \mathbbm{C} [z] \exists
! (R, \Lambda) \in \mathbbm{C} [z] \times \mathbbm{C}_n [z]$ such that
\begin{equation}
  P = T (R) + \Lambda . \label{eq:decP}
\end{equation}
By multiplying this equality by $\Omega / J^2$, and integrating we get
\begin{eqnarray}
  \int \frac{P \Omega}{J^2} & = & \frac{R \Mho}{J} + \int \frac{\Lambda
  \Omega}{J^2}, \nonumber\\
  \Gamma (P) & = & \frac{R \Mho}{J} + \Gamma (\Lambda) .  \label{eq:gammaP}
\end{eqnarray}
Since $R \Mho / J$ is algebraic, $\Gamma (P)$ is algebraic, iff $\Gamma
(\Lambda)$ is algebraic. For the study of this problem we get the following.

\begin{theorem}
  \label{th:redint}Let $P, \Omega, J$ be as above with $e_0$ and $e_1$ in
  $\mathbbm{Q} \backslash \mathbbm{Z}$.
  \begin{enumeratenumeric}
    \item The integral $\Gamma (P)$ is algebraic iff it belongs to the field
    $\mathbbm{C} (z) [\Omega]$, that is iff there exists $R \in \mathbbm{C}
    (z)$ such that $P = T (R)$.
    
    \item $\forall P \in \mathbbm{C} [z]$ and $R \in \mathbbm{C} (z)$, the
    relation $P = T (R)$ implies that $R \in \mathbbm{C} [z]$.
    
    \item For a polynomial $P \in \mathbbm{C} [z]$, let $R$ and $\Lambda$
    satisfy equation (\ref{eq:decP}). Then $\Gamma (P)$ is algebraic iff
    $\Lambda = 0$. Two integrals $\Gamma (P_1)$ and $\Gamma (P_2)$ satisfy an
    Ostrowski relation: $\Gamma (P_1) + d \Gamma (P_2)$ is algebraic iff for
    the corresponding $\Lambda$s, we have $\Lambda_1 + d \Lambda_2 = 0$.
    
    \item If $P$ is a non zero polynomial with $\deg (P) \leqslant n$, then
    $\Gamma (P)$ is transcendental. Moreover. the $n + 1$ integrals $\int
    \frac{z^s \Omega}{J^2}$ with $0 \leqslant s \leqslant n \nocomma$ are
    algebraically independent.
  \end{enumeratenumeric}
\end{theorem}

The ideas behind this result are very closed to what we did in
{\cite{duval08}}.

\begin{proof}
  1. Since the exponents are rational but not integers, there exist a minimal
  integral power $d \geqslant 2$ such that $\Omega^d \in \mathbbm{C} (z)$. As
  a consequence, the field extension $\mathbbm{C} (z) [\Omega] =\mathbbm{C}
  (z) [\Mho] /\mathbbm{C} (z)$ is a Kummer extension of degree $d$. Now,
  $\Gamma (P)$ is algebraic iff it belongs to $\mathbbm{C} (z) [\Mho]$. It can
  therefore be expanded into the form
  \[ \Gamma (P) = \int \frac{P \Omega}{J^2} = \sum_{s = 1}^d \frac{R_s}{J}
     \Mho^s \text{ \hspace{1em} \tmop{with} \hspace{1em} } R_s \in \mathbbm{C}
     (z) . \]
  Taking derivative of the above equation, we obtain the following expression
  \[ \frac{P \Omega}{J^2} = \sum_{s = 1}^d \left( \frac{R_s'}{R_s} -
     \frac{J'}{J} + s \frac{\Mho'}{\Mho} \right) \frac{R_s}{J} \Mho^s \]
  Since in the right hand side, each coefficient of $\Mho^s$ is in
  $\mathbbm{C} (z)$, we must have
  \[ \left( \frac{R_1}{J} \Mho \right)' = 0 \hspace{2em} \text{\tmop{for}} s
     \geqslant 2, \]
  and
  \[ \frac{P \Omega}{J^2} = \left( \frac{R_1}{J} \Mho \right)' = \frac{T
     (R_1) \Omega}{J^2} \hspace{1em} \Longleftrightarrow \hspace{1em} P = T
     (R_1), \]
  and the claim follows.
  
  2. If $P = T (R)$ with $R \in \mathbbm{C} (z)$, then the function
  \[ z \longmapsto \Gamma (z) = \int \frac{P \Omega}{J^2} = \frac{R}{J} \Mho,
  \]
  is holomorphic in an arbitrary \ simply connected domain of $\mathbbm{C}
  \setminus \{0, 1, z_1, \ldots, z_n \}$. So, if $R$ has got a pole, it must
  belong to $\{0, 1, z_1, \ldots, z_n \}$. But for all $p \geqslant 1$ , the
  leading term of $T (1 / z^p)$ is given by
  \[ T (1 / z^p) = J (0)  [e_0 + 1 - p] / z^p \neq 0 \nocomma, \]
  $\text{\tmop{since}} e_0 \nin \mathbbm{Z} \text{\tmop{and}} J (0) \neq 0.$As
  a consequence, $z = 0$ cannot be a pole of $R$. A similar argument hold at
  $z = 1$, since the expansion of $T (1 / (z - 1)^p)$ begins with $J (1)  [p -
  e_0 - 1] / (z - 1)^p$.
  
  Around $z = z_i$, $\Gamma' (z)$ is of the form
  \[ \Gamma' (z) = a / (z - z_i)^2 + 0 / (z - z_i) + h_1 (z), \]
  with $h_1$ holomorphic. Indeed, because $\Gamma$ is algebraic $\Gamma'$ does
  not have residue. So,
  \[ \Gamma (z) = - a / (z - z_i) + h_2 (z) = R \Mho / J = Rh_3 (z) / (z -
     z_i), \]
  with $h_2$ and $h_3$ holomorphic around $z_i$. Since $h_3 (z_i) \neq 0$, $R$
  cannot have a pole at $z_i$. Hence, $R$ is polynomial.
  
  3. and 4. follow from the following observation. Let $\Lambda \in
  \mathbbm{C}_n [z]$ be such that $\Gamma (\Lambda)$ is algebraic. Point 1
  implies the existence of some $R \in \mathbbm{C} (z)$ with $\Lambda = T
  (R)$. But according to point 2, $R$ is a polynomial. If $R \neq 0$ then
  $\deg (\Lambda) = \deg (R) + n + 1 > n$ which is contradictory. So $\Lambda
  = T (0) = 0$ and the integral is algebraic iff $\Lambda = 0$.
\end{proof}

\subsubsection{Linear forms and the equation \ $\Lambda = \Lambda (P) = 0$}

The above result show that the algebraicity of an integral $\Gamma (P)$
reduces to the vanishing of the polynomial $\Lambda$ appearing in equation
(\ref{eq:decP}). Although the correspondence $P \mapsto \Lambda$ is linear,
the decomposition given by equation (\ref{eq:decP}), is very hard to perform
effectively. Here, we are going to show that the vanishing of $\Lambda$ can be
controlled by the vanishing of some linear forms on $P$ which can be directly
computed thanks to some definite integrals.

In the most simple case, that is for $J = 1$, i.e,, when $J ( z)$ is a
constant, this is achieved thanks to the following.

\begin{proposition}
  \label{prop:testphi}Let $e_0$ and $e_1$ be two real numbers greater than $-
  1$, and \ belonging to $\mathbbm{Q} \backslash \mathbbm{Z}$ with $e_0 + e_1
  \nin \mathbbm{Z}$. Let us set $\Omega \assign z^{e_0}  (1 - z)^{e_1}$. Then
  we have
  \begin{enumeratenumeric}
    \item For any polynomial $P$, the primitive integral $\Phi \assign \int P
    \Omega$ is algebraic iff
    \[ \mu (P) \assign \int^1_0 P \Omega (z) dz = 0. \]
    \item For all $n \in \mathbbm{N}$,
    \[ \mu (z^n) = \frac{(e_0 + 1)_n}{(e_0 + e_1 + 2)_n} B (e_0 + 1 ; e_1 +
       1), \]
    where $B (p ; q)$ is the usual Euler Beta function and $(x)_n$ is the
    Pochammer symbol.
    
    \item If $P (z) = \sum p_n z^n$, then $\Phi$ is algebraic iff
    \[ \sum p_n  \frac{(e_0 + 1)_n}{(e_0 + e_1 + 2)_n} = 0 \]
  \end{enumeratenumeric}
\end{proposition}

The condition on the two exponents to be greater than -1, guaranties the
convergence of the generalised integrals between 0 and 1. The first point
shows that the set of polynomials for which $\Phi$ is algebraic is an
hyperplane given by the kernel of the linear form $\mu$.

Points 2 and 3, give explicit criterion on the coefficients of $P$ to decide
whether or not $\Phi$ is algebraic.

\begin{proof}
  1. Since $J = 1$, in the decomposition given by equation \ (\ref{eq:decP}):
  $P = T (R) + \Lambda$, we have that $\Lambda$ is a number. The corresponding
  relation (\ref{eq:gammaP}), can be written
  \[ \Phi = \int P \Omega = R \Mho + \Lambda \int \Omega = R (z) z^{e_0 + 1} 
     (1 - z)^{e_1 + 1} + \Lambda \int \Omega . \]
  Now let us \ compute $\Lambda$ by evaluating the integrals between 0 and 1.
  Since $e_p + 1 > 0$ for $p \in \{0 ; 1\}$, and $R \in \mathbbm{C} [z]$, we
  have
  \[ \mu (P) = \int_0^1 P \Omega (z) dz = \Lambda \int_0^1 \Omega (z) dz
     \hspace{1em} \Longrightarrow \hspace{1em} \Lambda = \mu (P) / \int_0^1
     \Omega (z) dz. \]
  2. Is a direct consequence of the relation
  \[ \mu (z^n) = \int_0^1 z^{e_0 + n}  (1 - z)^{e_1} dz = B (e_0 + n + 1 ;
     e_1 + 1) = \frac{\Gamma (e_0 + n + 1) \Gamma (e_1 + 1)}{\Gamma (e_0 + e_1
     + n + 2)} . \]
  3. Follows directly from the previous considerations.
\end{proof}

On the basis of the same ideas we now treat the case when $\deg (J) = n
\geqslant 1$.

Let $0 < z_1 \leqslant \cdots \leqslant z_n < 1$ be the roots of $J$. Let
$\gamma_0$ be a half of the circle going counterclockwise from 0 to 1. Let
$\gamma_i$ for $1 \leqslant i \leqslant n$ be some small trigonometric circles
each enclosing $z_i$ and no other root $z_j$. Let us \ consider the $n + 1$
linear forms on $\mathbbm{C} [z]$ given by
\[ \mathcal{L}_i (P) \assign \int_{\gamma_i} \frac{P \Omega}{J^2},
   \hspace{2em} 0 \leqslant i \leqslant n. \]
Then we get the following

\begin{proposition}
  \label{prop:testpsi}With the previous notations
  \begin{enumeratenumeric}
    \item If $P = T (R) + \Lambda$ as in relation (\ref{eq:decP}), then for
    all $0 \leqslant i \leqslant n$, $\mathcal{L}_i (P) =\mathcal{L}_i
    (\Lambda)$.
    
    \item The $n + 1$ linear forms $\mathcal{L}_i$ are free and $\int \frac{P
    \Omega}{J^2}$ is algebraic iff for all $0 \leqslant i \leqslant n$,
    $\mathcal{L}_i (P) = 0$.
  \end{enumeratenumeric}
\end{proposition}

Let us observe that this property is a generalisation of the previous one.
Indeed, $\mathcal{L}_0 = \mu$, for $\deg (J) = 0$, that is when $J = 1$. Here
the problem is that we do not find comparable simple closed formulae for the
linear forms $\mathcal{L}_i$. We mention this difficulty because the
$\mathcal{L}_i (P)$ got the flavour of some periods on some Abelian variety.
But we did not find this link precisely. This is probably the deep reason why
things are so complicated in our context. Maybe we did not find the proper
geometric space where the actual notions would get some more transparent
meaning.

\begin{proof}
  1. This is a direct consequence of relation (\ref{eq:gammaP}).
  
  2. For $P_0 = J^2$ the function $P \Omega / J^2 = \Omega$ has no residue at
  none of the $z_i$, hence
  \[ \mathcal{L}_0 (J^2) = \mu (1) \neq 0 \text{\hspace{1em} \tmop{and}
     \hspace{1em} } \mathcal{L}_i (J^2) = 0, \hspace{1em} \text{\tmop{for}
     \hspace{1em}} 1 \leqslant i \leqslant n . \]
  Now let us \ set $P_s \assign J^2 / (z - z_s)$, $\tmop{for} 1 \leqslant s
  \leqslant n$. The function $P_s \Omega / J^2 = \Omega / (z - z_s)$ has a non
  zero residue at $z_s$, \ and zero residue elsewhere. So, \ $\mathcal{L}_s
  (P_s) \neq 0$ and $\mathcal{L}_i (P_s) = 0$ for $i \neq s,$ and $1 \leqslant
  i \leqslant n$. From this it follows immediately that the $n + 1$ linear
  forms $\mathcal{L}_s$ are free on $\mathbbm{C} [z]$. According to point 1,
  their restriction to $\mathbbm{C}_n [z]$ form a basis of the dual space
  $\mathbbm{C}_n [z]^{\ast}$. As a consequence, any $\Lambda \in \mathbbm{C}_n
  [z]$ is zero iff it belongs to the common kernel of the linear forms. And we
  can therefore \ conclude that $\int \frac{P \Omega}{J^2}$ is algebraic iff
  for all $0 \leqslant i \leqslant n$, $\mathcal{L}_i (P) = 0$ according to
  Theorem \ref{th:redint}.
\end{proof}

\section{Reducing the virtual Abelianity of $\tmop{VE}_{2, \alpha}^{\gamma}$
and $\tmop{EX}_{2, \alpha, \beta}^{\gamma}$ to Ostrowski relations }

In this section we exhibit the integrals and the Ostrowski relations which are
going to govern the virtual Abelianity of the systems $\tmop{VE}_{2,
\alpha}^{\gamma} $ and \ $\tmop{EX}_{2, \alpha, \beta}^{\gamma}$. Next in the
two sections that follow we will test effectively these results.

\subsection{Getting obstruction thanks to the integrals $\Phi_{\nu}$ for $\nu
\in \{\alpha, \beta, \gamma\}$}

\begin{proposition}
  \label{prop:test1}With the notation of Section \ref{sec:hierarchy}, we get
  the following necessary conditions
  \begin{enumeratenumeric}
    \item If the differential Galois group of $\tmop{VE}_{2, \alpha}^{\gamma}$
    is virtually Abelian, then, we get two Ostrowski relations $\Phi +
    d_{\gamma} I_{\gamma}$ and $\Phi + d_{\alpha} I_{\alpha}$ are algebraic
    over $\mathbbm{C} (z)$ for some constants $d_{\gamma}$ and $d_{\alpha}$ in
    $\mathbbm{C}$. Here, $\Phi = \int \omega y_1 x^2_1$.
    
    \item If he differential Galois group of $\tmop{EX}_{2, \alpha,
    \beta}^{\gamma}$ is virtually Abelian, then we get three Ostrowski
    relations $\Phi + d_{\gamma} I_{\gamma}$, $\Phi + d_{\beta} I_{\beta}$ and
    $\Phi + d_{\alpha} I_{\alpha}$ are algebraic over $\mathbbm{C} (z)$. Here,
    $\Phi = \int \omega u_1 y_1 x_1$.
  \end{enumeratenumeric}
\end{proposition}

\begin{proof}
  Let \ $F / K$be the Picard-Vessiot extension $\tmop{VE}^{\gamma}_{2,
  \alpha}$. According to Section \ref{sec:hierarchy}, we get the following
  inclusion
  \[ K \subset F_1^{\gamma} : = K (I_{\gamma}) \subset F_2^{\gamma} \assign
     F_1^{\gamma} (\Phi_{\gamma}) \subset F, \]
  and a similar one by changing $\gamma$ to $\alpha$. The fact that $F / K$ is
  virtually Abelian implies the same property for $F_2^{\gamma} / K,$ and
  $F_2^{\alpha} / K$. Then we can conclude thanks to Lemma \ref{lem:fund}.
  Similar arguments hold when dealing with $\tmop{EX}_{2, \alpha,
  \beta}^{\gamma}$.
\end{proof}

\subsection{Strategy and Game}

Two cases may a priory happen, when applying the above Proposition: $\Phi$ is
either transcendental either algebraic over $K$ or $\mathbbm{C} (z)$.

\subsubsection{When $\Phi$ is transcendental}

If this occurs, then the virtual Abelianity of the differential Galois group
of $S = \tmop{VE}_{2, \alpha}^{\gamma}$, or $S = \tmop{EX}_{2, \alpha,
\beta}^{\gamma}$ implies that the corresponding constants $d_{\gamma} $,
$d_{\beta}$ and $d_{\alpha}$ are non zero complex numbers. As a first
consequence we must also get Ostrowski relations between the integrals $I$.
For instance, $d_{\gamma} I_{\gamma} - d_{\alpha} I_{\alpha}$ is algebraic...

\subsubsection{When $\Phi$ is algebraic}

In this situation, the Ostrowski relations of Proposition \ref{prop:test1},
hold with $d_{\gamma} = d_{\beta} = d_{\alpha} = 0$ and the proposition is
helpless. Unfortunately, as we shall see in Section 5 and 6 below, $\Phi$ is
algebraic very oftenly. This is the reason why we have to find new necessary
conditions for the virtual Abelianity of differential Galois groups of
$\tmop{VE}_{2, \alpha}^{\gamma}$ and $\tmop{EX}_{2, \alpha, \beta}^{\gamma}$
in this case. This will be the purpose of the next subsection.

\subsection{Getting obstruction when $\Phi$ is algebraic}

In this subsection, we give the two criteria such that groups of \
$\tmop{VE}_{2, \alpha}^{\gamma}$ and $\tmop{EX}_{2, \alpha, \beta}^{\gamma}$
are virtually Abelian when $\Phi$ is algebraic.

\subsubsection{Integration by part and new second level integrals $\Psi_m$
associated to the $\Phi_m$}

Here, we keep notations and formulae given in Section \ref{sec:hierarchy}, \
and we assume that for each system \ $S = \tmop{VE}_{2, \alpha}^{\gamma}$ or
$S = \tmop{EX}_{2, \alpha, \beta}^{\gamma}$ its corresponding $\Phi$ is
algebraic, i.e., $\Phi \in K$. We know that the Picard-Vessiot extension $F /
F_1 = \tmop{PV} (S) / F_1$ is generated by second level integrals
$\Phi_{\nu}$, $\Phi_{i, j}$ and $\Phi_{i, j, l}$ which have, according to the
given hierarchy, one, two or three indices. For each multi-index $m \in \{\nu
; (i, j) ; (i, j, l)\}$, let us symbolically write $\Phi_m = \int \Phi' I^m$.
Integration by part gives
\[ \Phi_m = \Phi I^m - \Psi_m \text{ \hspace{1em} \tmop{with}} \hspace{1em}
   \Psi_m \assign \int \Phi (I^m)' . \]
Since $\Phi \in K$, $\Phi I^m \in T (F_1 / K) = K [I]$, and $F / F_1 =
\tmop{PV} (S) / F_1$ is generated by the corresponding $\Psi_m$.

Our motivation to introduce these new integrals $\Psi_m$ is the fact that \
the computation of their associated cocycles $\sigma (\Psi_m) - \Psi_m$ is \
simpler than for the cocycles $\sigma (\Phi_m) - \Phi_m$.

Now, we give the precise formulae for the given $\Psi_m$ for one, two and
three indices respectively. Next, we compute their cocycle, and give a general
property about some specific cocycles.

\paragraph{For one index}

Here we have $\Psi_{\nu} \assign \int \Phi I'_{\nu}$, for $\nu \in \{\alpha,
\beta, \gamma\}$.

\begin{remark}
  \label{rem:psi} Let us \ observe that the integrals
  \[ \Psi \assign \int \Phi I', \]
  are not well defined objects. This is because, as a primitive integral,
  $\Phi$ is only defined up to an additive constant. Therefore, two integrals
  $\Psi_1$ and $\Psi_2$ defined for the same ``$\Phi$'' modulo constant terms,
  and the same $I$, are related by a relation of the form
  \[ \Psi_2 = \Psi_1 + dI + e, \]
  where $(d, e) \in \mathbbm{C}^2$. Hence, having an Ostrowski relation
  $\Psi_1 + dI \in K$ is therefore equivalent to have a representative $\Psi_2
  = \Psi_1 + dI$ which is algebraic. We will therefore use both of the two
  expressions.
  
  Let us \ observe also that point (3) of Theorem \ref{prop:ve2test2} below is
  coherent when the two representatives $\Psi_{\gamma}$ and $\Psi_{\alpha}$
  are defined with respect to the same $\Phi$.
  
  Observe also that the $\Psi_{\nu}$ are integrals of first level with respect
  to $K$, \ since $\Phi$, \ and the $I_{\nu}'$ are in $K$.
\end{remark}

\paragraph{For two indices}

For $i \neq j$, we have
\begin{eqnarray*}
  \Psi_{i, j} & = & \int \Phi (I_i' I_j + I_j' I_i) = \int \Psi_i' I_j +
  \Psi_j' I_i,\\
  \Psi_{i, j} & = & \Psi_i I_j + \Psi_j I_i - M_{i, j} \text{\hspace{1em}
  \tmop{with} \hspace{1em}} M_{i, j} \assign \int \Psi_i I'_j + \Psi_j I'_i .
\end{eqnarray*}
In the particular case where $i = j = \alpha$, by simplicity we divide by two
the original $\Psi$ by setting
\begin{eqnarray*}
  \Psi_{2 \alpha} & = & \int \Phi I'_{\alpha} I_{\alpha} = \int \Psi_{\alpha}'
  I_{\alpha},\\
  \Psi_{2 \alpha} & = & \Psi_{\alpha} I_{\alpha} - X \text{\hspace{1em}
  \tmop{with}} \hspace{1em} X \assign \int \Psi_{\alpha} I_{\alpha}' .
\end{eqnarray*}

\paragraph{For three indices}

We get
\begin{eqnarray*}
  \Psi_{\gamma, 2 \alpha} & = & \int \Phi (I_{\gamma} I_{\alpha}^2)' = \int
  \Phi (I_{\gamma}' I_{\alpha}^2 + 2 I_{\alpha}' I_{\alpha} I_{\gamma}),\\
  \Psi_{\gamma, 2 \alpha} & = & \int \Psi_{\gamma}' I_{\alpha}^2 + 2
  \Psi_{\alpha}' I_{\alpha} I_{\gamma}',\\
  \Psi_{\alpha, \beta, \gamma} & = & \int \Phi (I_{\alpha} I_{\beta}
  I_{\gamma})' = \int \Phi (I'_{\alpha} I_{\beta} I_{\gamma} + I_{\alpha}
  I'_{\beta} I_{\gamma} + I_{\alpha} I_{\beta} I'_{\gamma}),\\
  \Psi_{\alpha, \beta, \gamma} & = & \int \Psi_{\alpha}' I_{\beta} I_{\gamma}
  + \Psi_{\beta}' I_{\gamma} I_{\alpha} + \Psi_{\gamma}' I_{\alpha} I_{\beta}
  .
\end{eqnarray*}

\paragraph{General formulae for the cocycles $\mathcal{C}_m (\sigma) \assign
\sigma (\Psi_m) - \Psi_m$ when $\sigma \in G^{\circ}$}

Here $G$ denotes the corresponding Galois group. In order to simplify
notations we may some time write \ $c_{\nu}$ instead of $c_{\nu} (\sigma)$ in
the relations $\sigma (I_{\nu}) = I_{\nu} + c_{\nu} (\sigma)$, for $\sigma \in
G^{\circ}$. We get the following formulae
\begin{eqnarray*}
  \mathcal{C}_{i, j} (\sigma) & = & c_i \Psi_j + c_j \Psi_j + l_{i, j}
  (\sigma),\\
  \mathcal{C}_{2 \alpha} (\sigma) & = & 2 c_{\alpha} \Psi_{\alpha} + l_{2
  \alpha} (\sigma),\\
  \mathcal{C}_{\gamma, 2 \alpha} (\sigma) & = & 2 c_{\alpha} \Psi_{\alpha,
  \gamma} + 2 c_{\gamma} \Psi_{2 \alpha} + c_{\alpha}^2 \Psi_{\gamma} + 2
  c_{\alpha} c_{\gamma} \Psi_{\alpha} + l_{\gamma, 2 \alpha} (\sigma),\\
  \mathcal{C}_{\alpha, \beta, \gamma} (\sigma) & = & c_{\alpha} (\sigma)
  \Psi_{\beta, \gamma} + c_{\beta} (\sigma) \Psi_{\gamma, \alpha} + c_{\gamma}
  (\sigma) \Psi_{\alpha, \beta} + c_{\alpha} c_{\beta} \Psi_{\gamma} +
  c_{\beta} c_{\gamma} \Psi_{\alpha} + c_{\gamma} c_{\alpha} \Psi_{\beta} +
  l_{\alpha, \beta, \gamma} (\sigma),
\end{eqnarray*}
where the respective $l_m$ are function from $G^{\circ}$ to $\mathbbm{C}.$

We \ make only \ the computation for $\mathcal{C}_{\gamma, 2 \alpha} (\sigma)$
since the others are similar.

Since $\Psi_{\gamma, 2 \alpha}' = \Phi (I_{\gamma}' I_{\alpha}^2 + 2
I_{\alpha}' I_{\alpha} I_{\gamma})$, we get
\begin{eqnarray*}
  \sigma (\Psi_{\gamma, 2 \alpha}') & = & \Phi I'_{\gamma}  (I^2_{\alpha} + 2
  c_{\alpha} I_{\alpha} + c^2_{\alpha}) + 2 \Phi I_{\alpha}'  (I_{\alpha} +
  c_{\alpha})  (I_{\gamma} + c_{\gamma}),\\
  \sigma (\Psi_{\gamma, 2 \alpha}') & = & \Psi_{\gamma, 2 \alpha}' + 2
  c_{\alpha} \Psi'_{\alpha, \gamma} + 2 c_{\gamma} \Psi'_{2 \alpha} +
  c_{\alpha}^2 \Psi'_{\gamma} + 2 c_{\alpha} c_{\gamma} \Psi'_{\alpha},\\
  \mathcal{C}_{\gamma, 2 \alpha} (\sigma) & = & 2 c_{\alpha} \Psi_{\alpha,
  \gamma} + 2 c_{\gamma} \Psi_{2 \alpha} + c_{\alpha}^2 \Psi_{\gamma} + 2
  c_{\alpha} c_{\gamma} \Psi_{\alpha} + l_{\gamma, 2 \alpha} (\sigma) .
\end{eqnarray*}
The last relation has been obtained by integrating the previous one. As a
consequence the function $l_{\gamma, 2 \alpha} : G^{\circ} \rightarrow
\mathbbm{C}$ has constant values, but it is far to being a group morphism.
This can be simply seen if we translate for $l$ the cocycle relation for
$\mathcal{C}$.

\paragraph{Cocycles of degree in $I$ and symmetric matrices}

When proving the two theorems below, we will get explicit polynomial
expressions of the above cocycles, since a necessary condition for the virtual
Abelianity of $G$ is going to be
\[ \forall \sigma \in G^{\circ}, \hspace{2em} \mathcal{C}_m (\sigma) \in T
   (F_1 / K) = K [I] \assign K [I_1, \ldots, I_n] . \]
In practice, these cocycles are going to be of degree one or two in $I$.
Precisely, let us assume that $F_1 / K = K (I_1, \ldots, I_n)$, is generated
by $n$ independent first level integrals. Let us denote by $I^T \assign (I_1,
\ldots, I_n)$ and $C^T (\sigma) \assign (c_1 (\sigma), \ldots, c_n (\sigma))$.

We say that a cocycle $\mathcal{C} (\sigma) = \sigma (\Psi) - \Psi${\tmem{ is
of degree one in}} $I$, if it has \ the form
\[ \mathcal{C} (\sigma) = C (\sigma)^T AI + C (\sigma)^T  \tilde{A} C (\sigma)
   + F^T C (\sigma) + l (\sigma), \]
where $A \in M_n (K)$ and $\tilde{A} \in M_n (K)$ are $n \times n$ matrices,
$F \in K^n$ and, $l$ is a constant valued function.

In the particular case where $A$ and $\tilde{A}$ are constant matrices (i.e.,
belong to $M_n (\mathbbm{C})$), the mapping $\sigma \mapsto C (\sigma)^T 
\tilde{A} C (\sigma)$ is a constant valued function. So, the general
expression of {\tmem{the \ degree one cocycles with constant matrix}} \ can be
more simply written
\[ \mathcal{C} (\sigma) = C (\sigma)^T AI + C (\sigma)^T F + l (\sigma) . \]
With these notations, we get the following property which simplifies the proof
of the theorems, enlightening a link between the abelianity of a group, \ and
the general idea of {\tmem{symmetry}} which is realised here by symmetric
matrices.

\begin{lemma}
  \label{lem:cocycle}Let us assume that $F_1 / K = K (I_1, \ldots, I_n)$, is
  generated by $n$ independent first level integrals over $K$.
  \begin{enumeratenumeric}
    \item Let
    \[ \mathcal{C} (\sigma) = C (\sigma)^T AI + C (\sigma)^T  \tilde{A} C
       (\sigma) + F^T C (\sigma) + l (\sigma) \]
    be a general cocycle of degree one in $I$. Then, $\mathcal{C}$ coincides
    with a coboundary iff $A$ is a symmetric matrix, and $A - 2 \tilde{A} \in
    M_n (\mathbbm{C})$.
    
    \item Let $\mathcal{C} (\sigma) = \sigma (\Psi) - \Psi = C (\sigma)^T AI +
    C (\sigma)^T F + l (\sigma)$ be a degree one cocycle with constant
    matrices. Then $\mathcal{C}$ coincides with a coboundary iff $A$ is
    symmetric. If it is the case, then the corresponding second level integral
    $\Psi$ can be computed in a closed form thanks to a quadratic expression
    of the form
    \[ \Psi = \frac{1}{2} I^T AI + F^T I + J \text{\tmop{with}} J' \in K. \]
    \item For $n = 1$ every degree one cocycle with constant matrix is a
    coboundary.
  \end{enumeratenumeric}
\end{lemma}

\begin{proof}
  (1) Since $\mathcal{C}$ is of degree one in $I$, if it coincides with a
  coboundary of the form $\Delta P + h (\sigma)$, then $P (I)$ must be
  quadratic in $I$. It can therefore be written into the form
  \[ P (I) = I^T SI + B^T I, \]
  for some symmetric matrix $S \in M_n (K)$, and $B \in K^n$. The relation
  $\mathcal{C} (\sigma) = \Delta P + h (\sigma) = P ( I + c) - P ( I) + h (
  \sigma)$, is therefore equivalent to having
  \[ I^T AC (\sigma) + C (\sigma)^T  \tilde{A} C (\sigma) + F^T C (\sigma) + l
     (\sigma) = 2 I^T SC (\sigma) + C (\sigma)^T SC (\sigma) + B^T C (\sigma)
     + h (\sigma) . \]
  For a fixed value of $\sigma$, both side of this equation are affine linear
  forms in $I$ with coefficients in $K$. Hence, we must have
  \[ AC (\sigma) = 2 SC (\sigma), \]
  for all $\sigma \in G^{\circ}$. Since $C (\sigma)$ span all $\mathbbm{C}^n$,
  when $\sigma \in G^{\circ}$, we have that $A = 2 S$ is symmetric. Moreover,
  the previous equation is reduced to
  \[ C (\sigma)^T  ( \tilde{A} - S) C (\sigma) + (F - B)^T C (\sigma) = h
     (\sigma) - l (\sigma) . \]
  By derivating both sides of this equation, we obtain
  \[ C (\sigma)^T  ( \tilde{A}' - S') C (\sigma) + (F' - B')^T C (\sigma) =
     0, \text{\hspace{1em} \tmop{for} \tmop{all}} \sigma \in G^{\circ} . \]
  This is therefore equivalent to $\tilde{A}' - S' = 0$ and $F' - B' = 0$.
  That is to having $A - 2 \tilde{A} \in M_n (\mathbbm{C})$. Conversely, if
  those two conditions are satisfied, we just have to choose $B = F$ to get
  the desired coboundary.
  
  (2) When $\mathcal{C}$ is of degree one with constant matrix, it coincides
  with a coboundary iff $A$ is constant and symmetric, since in this case
  there is no condition on $\tilde{A}$. Conversely, if $A$ is symmetric, the
  computation above with $\tilde{A} = 0$ \ shows that $\Psi$ and $\tilde{\Psi}
  \assign \frac{1}{2} I^T AI + F^T I$, have equal cocycles up to a constant
  valued function, therefore the difference $\Psi - \tilde{\Psi}$ is a first
  level integral.
  
  (3) is obvious since a $1 \times 1$ matrix is always symmetric.
\end{proof}

Let's observe that the computations of the second level integral $\Phi_1$
appearing in Lemma \ref{lem:fund}, is a particular case of points (2) and (3)
of the above lemma.

\subsubsection{The criteria for the virtual Abelianity of $\tmop{EX}_{2,
\alpha, \beta}^{\gamma}$ when $\Phi$ is algebraic}

We decided to begin with $\tmop{EX}_{2, \alpha, \beta}^{\gamma}$ because here,
the role played by the integrals $I_{\nu}$ is symmetric. This is not the case
when dealing with $\tmop{VE}_{2, \alpha}^{\gamma}$. As a consequence, although
they proceed with the same methods, the proof of Theorem \ref{prop:EXtest2} is
more transparent than the proof of Theorem \ref{prop:ve2test2} below.

For the statement and proof of the following result, we use the notations
introduced above. The first point of the theorem gives a necessary condition
for the virtual Abelianity of $\tmop{EX}_{2, \alpha, \beta}^{\gamma}$. Having
expressed this first necessary condition, the next three points will give
sufficient conditions for the virtual Abelianity of $\tmop{EX}_{2, \alpha,
\beta}^{\gamma}$. Each of them depends on the degree of dependence of the
integrals $I_{\gamma}$, $I_{\beta}$ and $I_{\alpha}$. Notice also that also
that in the formulae $G \subset K^p$, with $p = 2$ or 3 below, the letter $G$
denotes a $p$-component vector which has nothing to do with the Galois group.

\begin{theorem}
  \label{prop:EXtest2}With the notation of Section \ref{sec:hierarchy}, let us
  assume that $\Phi = \int \omega u_1 y_1 x_1$ is algebraic, i.e., $\Phi \in
  K$.
  \begin{enumeratenumeric}
    \item Let us \ denote by $\Psi = (\Psi_{\alpha}, \Psi_{\beta},
    \Psi_{\gamma})^T$, and similar notations for the three component vector
    $I$. If $\tmop{EX}_{2, \alpha, \beta}^{\gamma}$ has \ virtually Abelian
    differential Galois group, then there exists an Ostrowski relation between
    $\Psi$ and $I$ of the form :
    \[ \Psi = DI + F \in K^3, \]
    for some constant $3 \times 3$ matrix $D$ and $F \in K^3$.
    
    \item If the integrals $I_{\alpha}, I_{\beta}, I_{\gamma}$ are independent
    over $K$, then $\tmop{EX}_{2, \alpha, \beta}^{\gamma}$ has a virtually
    Abelian differential Galois group iff the two following conditions are
    satisfied
    \begin{enumeratealpha}
      \item There exists a unique determination of $\Phi$ modulo constants
      such that each $\Psi_{\mu}$ for $\mu \in \{\gamma ; \beta ; \alpha\}$ is
      algebraic. This correspond to having $D = 0$ in point (1).
      
      \item For this determination of $\Phi$ the integrals $M_{i, j}$ are of
      first level, and can be expanded into the form
      \[ M = EI + G, \]
      for some constant $3 \times 3$ symmetric matrix $E$, where $G \subset
      K^3$,and \ $M \assign (M_{\beta, \gamma}, M_{\gamma, \alpha}, M_{\alpha,
      \beta})^T$.
    \end{enumeratealpha}
    \item If the integrals $I_{\alpha}, I_{\beta}, I_{\gamma}$ form a system
    of rank one over $K$, that is if we have Ostrowski relations of the forms
    $I_{\beta} - \theta_{\beta} I_{\alpha} \in K$ and $I_{\gamma} -
    \theta_{\gamma} I_{\alpha} \in K$, then $\tmop{EX}_{2, \alpha,
    \beta}^{\gamma}$ has a virtually Abelian differential Galois group iff
    condition (1) holds and, for the first level integrals $N_{i, j}$ defined
    by equation (\ref{eq:rank1}) below, we get an Ostrowski relation of the
    form
    \[ N_{\beta, \gamma} + \theta_{\beta} N_{\gamma, \alpha} +
       \theta_{\gamma} N_{\alpha, \beta} = eI_{\alpha} + g, \]
    with $e \in \mathbbm{C}$ and $g \in K$.
    
    \item If the integrals $I_{\alpha}, I_{\beta}, I_{\gamma}$ form a system
    of rank two over $K$, that is if we get one Ostrowski relation of the form
    \[ I_{\gamma} - bI_{\beta} - aI_{\alpha} \in K, \]
    and $I_{\alpha}, I_{\beta}$ are independent. Then, $\tmop{EX}_{2, \alpha,
    \beta}^{\gamma}$ has a virtually Abelian differential Galois group iff the
    following conditions are satisfied
    \begin{enumeratealpha}
      \item There exists a choice of $\Phi$ modulo constants such that (1) can
      be written
      \[ \Psi_{\alpha} = xI_{\beta} + f_{\alpha}, \Psi_{\beta} = yI_{\alpha}
         + f_{\beta}, \Psi_{\gamma} + b \Psi_{\beta} + a \Psi_{\alpha} \in K.
      \]
      with $x$ and $y$ in $\mathbbm{C}$.
      
      \item For the first level integrals $N_{i, j}$ defined by equation
      (\ref{eq:rank2}) below, we have a relation of the form
      \[ \left(\begin{array}{c}
           N_{\beta, \gamma} + aN_{\alpha, \beta}\\
           N_{\alpha, \gamma} + bN_{\alpha, \beta}
         \end{array}\right) = E \left(\begin{array}{c}
           I_{\alpha}\\
           I_{\beta}
         \end{array}\right) + G. \]
      where, $E$ is a $2 \times 2$ constant symmetric matrix and $G \in K^2$.
    \end{enumeratealpha}
  \end{enumeratenumeric}
\end{theorem}

\begin{proof}
  (1) Let us set $F / K = \tmop{PV} (\tmop{EX}_{2, \alpha, \beta}^{\gamma}) /
  K,$ and assume that $G$ is virtually Abelian. According to Theorem
  \ref{thliouville}, since each $\Psi_{i, j} \in F$, the cocycles
  \[ \mathcal{C}_{i, j} (\sigma) = \sigma (\Psi_{i, j}) - \Psi_{i, j} = c_i
     (\sigma) \Psi_j + c_j (\sigma) \Psi_i + l_{i, j} (\sigma) \in T (F_1 / K)
     = K [I] \assign K [I_{\gamma}, I_{\beta}, I_{\alpha}] . \]
  By considering all those possible relations with $i \neq j$, we get three
  relations which can be written into matrix form
  \[ \left(\begin{array}{ccc}
       0 & c_{\gamma} (\sigma) & c_{\beta} (\sigma)\\
       c_{\gamma} (\sigma) & 0 & c_{\alpha} (\sigma)\\
       c_{\beta} (\sigma) & c_{\alpha} (\sigma) & 0
     \end{array}\right)  \left(\begin{array}{c}
       \Psi_{\alpha}\\
       \Psi_{\beta}\\
       \Psi_{\gamma}
     \end{array}\right) = C \Psi \in T (F_1 / K)^3 . \]
  But $\det (C) = 2 c_{\alpha} (\sigma) c_{\beta} (\sigma) c_{\gamma} (\sigma)
  \neq 0$, for some $\sigma \in G^{\circ}$. So, by inverting this system we
  get that all $\Psi_{\mu} \in T (F_1 / K)$. So, $\Psi_{\mu}, I_{\alpha},
  I_{\beta}, I_{\gamma}$ are four dependant integrals of first level over $K$,
  from which we can deduce the desired Ostrowski relations.
  
  (2) By plugging each $\Psi_i = \sum_{\mu} d_{i, \mu} I_{\mu} + f_{\mu}$ into
  the above $\mathcal{C}_{i, j} (\sigma)$, we get that each $\sigma (\Psi_{i,
  j}) - \Psi_{i, j}$ is a cocycle of degree one with constant matrix. Indeed,
  let us do this for $\Psi_{\alpha, \beta}$:
  \begin{eqnarray*}
    c_{\alpha} \Psi_{\beta} + c_{\beta} \Psi_{\alpha} & = & (c_{\alpha},
    c_{\beta}, c_{\gamma})  \left(\begin{array}{ccc}
      d_{\beta, \alpha} & d_{\beta, \beta} & d_{\beta, \gamma}\\
      d_{\alpha, \alpha} & d_{\alpha, \beta} & d_{\alpha, \gamma}\\
      0 & 0 & 0
    \end{array}\right)  \left(\begin{array}{c}
      I_{\alpha}\\
      I_{\beta}\\
      I_{\gamma}
    \end{array}\right) + (c_{\alpha}, c_{\beta}, c_{\gamma}) 
    \left(\begin{array}{c}
      f_{\beta}\\
      f_{\alpha}\\
      0
    \end{array}\right)\\
    \mathcal{C}_{\alpha, \beta} (\sigma) & = & C (\sigma)^T A_{\alpha, \beta}
    I + C (\sigma)^T F_{\alpha, \beta} + l_{\alpha, \beta} (\sigma) .
  \end{eqnarray*}
  According to Theorem \ref{thliouville}, the Abelianity of $G^{\circ}$
  implies that the cocycle must be a coboundary. From Lemma \ref{lem:cocycle}
  the equivalent condition is that $A_{\alpha, \beta}$ is a symmetric matrix.
  As a consequence, we get that
  \[ d_{\beta, \gamma} = d_{\alpha, \gamma} = 0 \text{ \hspace{1em}
     \tmop{and} \hspace{1em} } d_{\alpha, \alpha} = d_{\beta, \beta} . \]
  Since the same arguments hold by considering $\mathcal{C}_{\beta, \gamma}$
  and $\mathcal{C}_{\gamma, \alpha}$, we deduce that $D = d \Iota_3$, for some
  $d \in \mathbbm{C}$. As a consequence, we can write
  \[ \Psi_{\mu} = \int \Phi I_{\mu}' = dI_{\mu} + f_{\mu} \hspace{1em}
     \Longrightarrow \hspace{1em} \int (\Phi - d) I_{\mu}' = f_{\mu} \in K, \]
  for \ $\mu \in \{\alpha, \beta, \gamma\}$. Therefore, if we substitute,
  $\Phi - d$ to $\Phi$ in the definition of the corresponding $\Psi_{\mu}$,
  then they all are algebraic. This proves point (a). Let us assume from now
  that we are in this situation. Then, each $M_{i, j} = \int \Psi_i I'_j +
  \Psi_j I'_i$ is a first level integral.
  
  Since, $\Psi_{\alpha, \beta, \gamma} = \int \Psi'_{\gamma} I_{\beta}
  I_{\alpha} + \Psi'_{\beta} I_{\gamma} I_{\alpha} + \Psi'_{\alpha} I_{\gamma}
  I_{\beta}$ is a second level integral that belongs to $F$ its associated
  cocycles $\mathcal{C}_{\alpha, \beta, \gamma} (\sigma) \in T (F_1 / K)$ for
  all $\sigma \in G^{\circ}$. But
  \[ \mathcal{C}_{\alpha, \beta, \gamma} (\sigma) = c_{\alpha} (\sigma)
     \Psi_{\beta, \gamma} + c_{\beta} (\sigma) \Psi_{\gamma, \alpha} +
     c_{\gamma} (\sigma) \Psi_{\alpha, \beta} + c_{\alpha} c_{\beta}
     \Psi_{\gamma} + c_{\beta} c_{\gamma} \Psi_{\alpha} + c_{\gamma}
     c_{\alpha} \Psi_{\beta} + l_{\alpha, \beta, \gamma} (\sigma) . \]
  As a consequence, each $\Psi_{i, j} \in T (F_1 / K)$ since the vectors $C
  (\sigma) = (c_{\alpha} (\sigma), c_{\beta} (\sigma), c_{\gamma} (\sigma))^T$
  ranges $\mathbbm{C}^3$, when $\sigma$ ranges $G^{\circ}$. Therefore, each
  $M_{i, j} = \Psi_i I_j + \Psi_j I_i - \Psi_{i, j}$ is a first level integral
  which belongs to $T (F_1 / K)$. Hence, we get a matrix type Ostrowski
  relation of the form $M = EI + G$, where $E$ is constant $3 \times 3$
  matrix, and $G \in K^3$. The previous relation implies that
  \[ Z \assign \left(\begin{array}{c}
       \Psi_{\beta, \gamma}\\
       \Psi_{\gamma, \alpha}\\
       \Psi_{\alpha, \beta}
     \end{array}\right) = YI - EI - G, \text{ \hspace{1em} \tmop{with}
     \hspace{1em} } Y \assign \left(\begin{array}{ccc}
       0 & \Psi_{\gamma} & \Psi_{\beta}\\
       \Psi_{\gamma} & 0 & \Psi_{\alpha}\\
       \Psi_{\beta} & \Psi_{\alpha} & 0
     \end{array}\right) \in M_3 (K) . \]
  As a consequence, we get the following formulae for the cocycle
  \begin{eqnarray*}
    \mathcal{C}_{\alpha, \beta, \gamma} (\sigma) & = & C (\sigma)^T Z +
    \frac{1}{2} C (\sigma)^T YC (\sigma) + l_{\alpha, \beta, \gamma}
    (\sigma)\\
    & = & C (\sigma)^T  (YI - EI - G) + \frac{1}{2} C (\sigma)^T YC (\sigma)
    + l_{\alpha, \beta, \gamma} (\sigma),\\
    \mathcal{C}_{\alpha, \beta, \gamma} (\sigma) & = & C (\sigma)^T  (Y - E) I
    + \frac{1}{2} C (\sigma)^T YC (\sigma) - C (\sigma) ^T G + l_{\alpha,
    \beta, \gamma} (\sigma) .
  \end{eqnarray*}
  Hence, $\mathcal{C}_{\alpha, \beta, \gamma}$ is a general coboundary of
  degree one. We get the formula of the first point of Lemma \ref{lem:cocycle}
  by setting
  \[ A \assign Y - E \text{\tmop{and}} \tilde{A} \assign Y / 2. \]
  Since $Y$ is symmetric, $\mathcal{C}_{\alpha, \beta, \gamma}$ is a
  coboundary iff $E$ is a symmetric matrix. This proves the claim.
  
  (3) For simplicity, we set $I \assign I_{\alpha}$ and $c \assign
  c_{\alpha}$, then by assumption we get $\forall \nu \in \{\alpha, \beta,
  \gamma\}$, $c_{\nu} = \theta_{\nu} c$, with $\theta_{\alpha} = 1$. Let's
  assume that $G$ is virtually Abelian. Here, the Ostrowski relations of point
  (1) can be written
  \[ \forall \nu \in \{\alpha, \beta, \gamma\}, \Psi_{\nu} = d_{\nu} I +
     f_{\nu} \text{\tmop{with}} d_{\nu} \in \mathbbm{C}, f_{\nu} \in K. \]
  As a consequence, the cocycles $\mathcal{C}_{i, j}$ are of degree one with
  constant $1 \times 1$ matrices. Indeed we get
  \begin{eqnarray*}
    c_i (\sigma) \Psi_j + c_j (\sigma) \Psi_i & = & c (\theta_i \Psi_j +
    \theta_j \Psi_i)\\
    \mathcal{C}_{i, j} (\sigma) = \sigma (\Psi_{i, j}) - \Psi_{i, j} & = & c
    (\theta_i d_j + \theta_j d_i) I + c (\theta_i f_j + \theta_j f_i) + l_{i,
    j} (\sigma) .
  \end{eqnarray*}
  Therefore, according to Lemma \ref{lem:cocycle}, those cocycles are
  coboundary and there exist some first level integrals $N_{i, j}$ such that
  \begin{equation}
    \Psi_{i, j} = \frac{1}{2}  (\theta_i d_j + \theta_j d_i) I^2 + (\theta_i
    f_j + \theta_j f_i) I + N_{i, j} . \label{eq:rank1}
  \end{equation}
  Now, if we denote by the symbol $\bigoplus$the sum over the three cyclic
  permutations of the indices $(\alpha, \beta, \gamma)$, we get
  \begin{eqnarray*}
    \mathcal{C}_{\alpha, \beta, \gamma} (\sigma) & = & \bigoplus c_h \Psi_{i,
    j} + \bigoplus c_i c_j \Psi_h + l_{\alpha, \beta, \gamma} (\sigma),\\
    \mathcal{C}_{\alpha, \beta, \gamma} (\sigma) & = & c \bigoplus \left(
    \frac{\theta_h}{2} (\theta_i d_j + \theta_j d_i) I^2 + \theta_h (\theta_i
    f_j + \theta_j f_i) I + \theta_h N_{i, j} \right)\\
    &  & + c^2 \bigoplus \theta_i \theta_j  (d_h I + f_h) + l_{\alpha, \beta,
    \gamma} (\sigma),\\
    \mathcal{C}_{\alpha, \beta, \gamma} (\sigma) & = & d [cI^2 + c^2 I] + 2
    cfI + cf + c^2 f + cN + l_{\alpha, \beta, \gamma} (\sigma),
  \end{eqnarray*}
  where we have set $d \assign \bigoplus \theta_i \theta_j d_h \in
  \mathbbm{C}$, $f \assign \bigoplus \theta_h \theta_i f_j \in K$ and $N
  \assign \bigoplus \theta_h N_{i, j}$ is a first level integral. Since
  $\mathcal{C}_{\alpha, \beta, \gamma} (\sigma) \in T (F_1 / K)$, for all
  $\sigma \in G^{\circ}$, $N$ also belongs to $T (F_1 / K)$. As a consequence,
  we get an Ostrowski relation of the form
  \[ N = N_{\beta, \gamma} + \theta_{\beta} N_{\gamma, \alpha} +
     \theta_{\gamma} N_{\alpha, \beta} = eI_{\alpha} + g = eI + g. \]
  This is precisely our additional necessary condition. Conversely, let's
  assume that (1) hold and $N = eI + g$. We already saw that all the $\Psi_{i,
  j}$ can be computed in closed form. We will conclude by showing that in fact
  \ $\mathcal{C}_{\alpha, \beta, \gamma}$ is a coboundary. By plugging $N = eI
  + g$ inside the last expression of $\mathcal{C}_{\alpha, \beta, \gamma}$, we
  see that the latter is of degree two in $I$. But we are going to decrease
  its degree thanks to the following trick
  \[ \Delta (I^3 / 3) = (I + c)^3 / 3 - I^3 / 3 = cI^2 + c^2 I + c^3 / 3. \]
  Therefore,
  \begin{eqnarray*}
    \mathcal{C}_{\alpha, \beta, \gamma} (\sigma) - \Delta (dI^3 / 3) & = & 2
    cfI + cf + c^2 f + cN + l_{\alpha, \beta, \gamma} (\sigma) - c^3 / 3\\
    & = & \Delta (fI^2) + ceI + cg + l_{\alpha, \beta, \gamma} (\sigma) - c^3
    / 3\\
    \mathcal{C}_{\alpha, \beta, \gamma} (\sigma) & = & \Delta [dI^3 / 3 + (f +
    e / 2) I^2 + gI] + l (\sigma) .
  \end{eqnarray*}
  So, $\mathcal{C}_{\alpha, \beta, \gamma}$ is a coboundary and $G$ is
  virtually Abelian.
  
  (4) Here we have $c_{\gamma} = bc_{\beta} + ac_{\alpha}$. If we set: $I^T
  \assign (I_{\alpha}, I_{\beta})$ and $C (\sigma)^T \assign (c_{\alpha},
  c_{\beta})$ then the Ostrowski relation of point (1) can be written
  \[ \Psi = DI + F = \left(\begin{array}{c}
       \Psi_{\alpha}\\
       \Psi_{\beta}\\
       \Psi_{\gamma}
     \end{array}\right) = \left(\begin{array}{cc}
       d_{\alpha, \alpha} & d_{\alpha, \beta}\\
       d_{\beta, \alpha} & d_{\beta, \beta}\\
       d_{\gamma, \alpha} & d_{\gamma, \beta}
     \end{array}\right)  \left(\begin{array}{c}
       I_{\alpha}\\
       I_{\beta}
     \end{array}\right) + \left(\begin{array}{c}
       f_{\alpha}\\
       f_{\beta}\\
       f_{\gamma}
     \end{array}\right) . \]
  Again, the cocycles $\mathcal{C}_{i, j}$ are going to be of degree one with
  constant matrices of size $2 \times 2$ this time. If we write
  $\mathcal{C}_{i, j} (\sigma) = C (\sigma)^T A_{i, j} I + C (\sigma)^T F_{i,
  j} + l_{i, j} (\sigma)$, the same computations as in point (3) give
  \begin{eqnarray*}
    A_{\beta, \gamma} = \left(\begin{array}{cc}
      ad_{\beta, \alpha} & ad_{\beta, \beta}\\
      d_{\gamma, \alpha} + bd_{\beta, \alpha} & d_{\gamma, \beta} + bd_{\beta,
      \beta}
    \end{array}\right), & F_{\beta, \gamma} = \left(\begin{array}{c}
      af_{\beta}\\
      f_{\gamma} + bf_{\beta}
    \end{array}\right), & \\
    A_{\alpha, \gamma} = \left(\begin{array}{cc}
      d_{\gamma, \alpha} + ad_{\alpha, \alpha} & d_{\gamma, \beta} +
      ad_{\alpha, \beta}\\
      bd_{\alpha, \alpha} & bd_{\alpha, \beta}
    \end{array}\right), & F_{\alpha, \gamma} = \left(\begin{array}{c}
      f_{\gamma} + af_{\alpha}\\
      bf_{\alpha}
    \end{array}\right), & \\
    A_{\alpha, \beta} = \left(\begin{array}{cc}
      d_{\beta, \alpha} & d_{\beta, \beta}\\
      d_{\alpha, \alpha} & d_{\alpha, \beta}
    \end{array}\right), & F_{\alpha, \beta} = \left(\begin{array}{c}
      f_{\beta}\\
      f_{\alpha}
    \end{array}\right) . & 
  \end{eqnarray*}
  The three cocycles $\mathcal{C}_{i, j}$ are coboundary iff the matrices
  $A_{i, j}$ are symmetric. This translates to having $D$ of the form
  \[ D = \left(\begin{array}{cc}
       d_{\alpha, \alpha} & d_{\alpha, \beta}\\
       d_{\beta, \alpha} & d_{\beta, \beta}\\
       d_{\gamma, \alpha} & d_{\gamma, \beta}
     \end{array}\right) = \left(\begin{array}{cc}
       d & x\\
       y & d\\
       ad - by & bd - ax
     \end{array}\right) \in M_{3, 2} (\mathbbm{C}) . \]
  If we look at the first line of $D$ this gives
  \[ \Psi_{\alpha} = \int \Phi I_{\alpha}' = dI_{\alpha} + xI_{\beta} +
     f_{\beta} \Leftrightarrow \int (\Phi - d) I_{\alpha}' = 0 I_{\alpha} +
     xI_{\beta} + f_{\alpha} . \]
  As a consequence, if we change $\Phi$ to $\Phi - d$ in the definition of the
  $\Psi_{\nu}$, the new corresponding matrix $D$ will be simplified into the
  form
  \[ D = \left(\begin{array}{cc}
       0 & x\\
       y & 0\\
       - by & - ax
     \end{array}\right) \Rightarrow \left(\begin{array}{c}
       \Psi_{\alpha}\\
       \Psi_{\beta}\\
       \Psi_{\gamma}
     \end{array}\right) = \left(\begin{array}{c}
       xI_{\beta}\\
       yI_{\alpha}\\
       - byI_{\alpha} - axI_{\beta}
     \end{array}\right) + \left(\begin{array}{c}
       f_{\alpha}\\
       f_{\beta}\\
       f_{\gamma}
     \end{array}\right), \]
  and this relation is equivalent to condition (4.a).
  
  When this condition is satisfied, the matrices $A_{i, j}$ are symmetric,
  hence by the second point of the lemma, we get explicit formulae of the form
  \begin{equation}
    \Psi_{i, j} = \frac{1}{2} I^T A_{i, j} I + F_{i, j}^T I + N_{i, j},
    \label{eq:rank2}
  \end{equation}
  where the $N_{i, j}$ are first level integral. Precisely, by expressing the
  $A_{i, j}$ in terms of $x, y, \theta_{\alpha}, \theta_{\beta}$, we get the
  following expressions
  \begin{eqnarray*}
    A_{\beta, \gamma} = \left(\begin{array}{cc}
      ay & 0\\
      0 & - ax
    \end{array}\right) & \Longrightarrow & \Psi_{\beta, \gamma} = \frac{ay}{2}
    I_{\alpha}^2 - \frac{ax}{2} I_{\beta}^2 + (af_{\beta}) I_{\alpha} +
    (f_{\gamma} + bf_{\beta}) I_{\beta} + N_{\beta, \gamma},\\
    A_{\alpha, \gamma} = \left(\begin{array}{cc}
      - by & 0\\
      0 & bx
    \end{array}\right) & \Longrightarrow & \Psi_{\alpha, \gamma} = -
    \frac{by}{2} I_{\alpha}^2 + \frac{bx}{2} I_{\beta}^2 + (f_{\gamma} +
    af_{\alpha}) I_{\alpha} + (bf_{\alpha}) I_{\beta} + N_{\alpha, \gamma},\\
    A_{\alpha, \beta} = \left(\begin{array}{cc}
      y & 0\\
      0 & x
    \end{array}\right) & \Longrightarrow & \Psi_{\alpha, \beta} = \frac{y}{2}
    I_{\alpha}^2 + \frac{x}{2} I_{\beta}^2 + (f_{\beta}) I_{\alpha} +
    (f_{\alpha}) I_{\beta} + N_{\alpha, \beta} .
  \end{eqnarray*}
  Next, by plugging these expressions in closed form of $\Psi_{i, j}$ and
  $\Psi_{\nu}$ into $\mathcal{C}_{\alpha, \beta, \gamma}$ we get a formula of
  degree two in $I$ where $A$ is the symmetric matrix given by $A \assign
  \left(\begin{array}{cc}
    af_{\beta} & f / 2\\
    f / 2 & bf_{\alpha}
  \end{array}\right)$ with $f \assign f_{\gamma} + bf_{\beta} + af_{\alpha}$:
  \begin{eqnarray*}
    \mathcal{C}_{\alpha, \beta, \gamma} (\sigma) & = & ay [c_{\alpha}
    I_{\alpha}^2 + c_{\alpha}^2 I_{\alpha}] + bx [c_{\beta} I_{\beta}^2 +
    c_{\beta}^2 I_{\beta}] + l_{\alpha, \beta, \gamma} (\sigma)\\
    &  & + C (\sigma)^T 2 AI + C (\sigma)^T AC (\sigma)\\
    &  & + C (\sigma)^T \left(\begin{array}{c}
      N_{\beta, \gamma} + aN_{\alpha, \beta}\\
      N_{\alpha, \gamma} + bN_{\alpha, \beta}
    \end{array}\right) .
  \end{eqnarray*}
  In the first line above we recognise an expression of the form $\Delta (
  \frac{ay}{3} I_{\alpha}^3 + \frac{bx}{3} I_{\beta}^3) + l (\sigma)$.
  Moreover, since \ $\mathcal{C}_{\alpha, \beta, \gamma} (\sigma) \in T (F_1 /
  K) \forall \sigma \in G^{\circ}$, the first level integrals $N_{\beta,
  \gamma} + aN_{\alpha, \beta}$ and $N_{\alpha, \gamma} + bN_{\alpha, \beta}$
  are in $T (F_1 / K)$. So we get an Ostrowski relation of the form
  \[ \left(\begin{array}{c}
       N_{\beta, \gamma} + aN_{\alpha, \beta}\\
       N_{\alpha, \gamma} + bN_{\alpha, \beta}
     \end{array}\right) = EI + G \text{\tmop{with}} E \in M_2 (\mathbbm{C}), G
     \in K^2 . \]
  Hence,
  \[ \mathcal{C}_{\alpha, \beta, \gamma} (\sigma) = \Delta ( \frac{ay}{3}
     I_{\alpha}^3 + \frac{bx}{3} I_{\beta}^3) + C (\sigma)^T  (2 A + E) I + C
     (\sigma)^T AC (\sigma) + C (\sigma)^T G + l (\sigma), \]
  is a coboundary, iff $E$ is a constant $2 \times 2$ symmetric matrix. This
  proves the claim.
\end{proof}

\subsubsection{The criteria for the virtual Abelianity of $\tmop{VE}_{2,
\alpha}^{\gamma}$ when $\Phi$ is algebraic}

Here, we also use the notations of Section \ref{sec:hierarchy}. Again, the
first three points of the theorem below, give necessary conditions for the
virtual Abelianity of $\tmop{VE}_{2, \alpha}^{\gamma}$. Points 5 and 6, give
sufficient conditions according to the dependence of the two integrals
$I_{\alpha}$ and $I_{\gamma}$.

\begin{theorem}
  \label{prop:ve2test2}With the notation of Section \ref{sec:hierarchy}, let
  us assume that $\Phi = \int \omega y_1 x^2_1$ is algebraic (i.e. $\Phi \in
  K$). If $\tmop{VE}_{2, \alpha}^{\gamma}$ is virtually Abelian, then we get
  the following
  \begin{enumeratenumeric}
    \item There exists an Ostrowski relation between $\Psi_{\alpha}$ and
    $I_{\alpha}$ : $\Psi_{\alpha} - d_{\alpha} I_{\alpha} \in K$.
    
    \item There exists an Ostrowski relation between $\Psi_{\gamma}$,
    $I_{\gamma}$ and $I_{\alpha}$ : $\Psi_{\gamma} - d_{\gamma} I_{\gamma} -
    dI_{\alpha} \in K$.
    
    \item If in the previous relations $d_{\gamma} \neq d_{\alpha}$ then there
    exists an Ostrowski relation between $I_{\gamma}$ and $I_{\alpha}$.
    
    \item Conversely, if $\Phi \in K$ and conditions (1),(2) and (3) hold
    true, then independently of the virtual Abelianity of $\tmop{VE}_{2,
    \alpha}^{\gamma}$, the second level integrals $\Psi_{\alpha},
    \Psi_{\gamma}, \Psi_{2 \alpha}, \Psi_{\gamma, \alpha}$, $X$ and $M$ can be
    computed in closed form. Moreover, $\tmop{VE}_{2, \alpha}^{\gamma}$ is
    virtually Abelian iff $\Psi_{\gamma, 2 \alpha}$ can also be computed in
    closed form.
    
    \item If $\Phi \in K$ and conditions (1),(2) and (3) hold true with
    $I_{\alpha}$ and $I_{\gamma}$ independent. According to points (1) and
    (3), there is unique choice of $\Phi$ modulo constants such that
    $\Psi_{\alpha} \in K$ and $\Psi_{\gamma} = dI_{\alpha} + f$. Then
    $\tmop{VE}_{2, \alpha}^{\gamma}$ is virtually Abelian iff the two
    following conditions are satisfied:
    \begin{itemizedot}
      \item The integrals $X$ and $M$ have polynomial expressions of the form
      \[ \left(\begin{array}{c}
           M\\
           X
         \end{array}\right) = \left(\begin{array}{c}
           dI_{\alpha}^2 / 2\\
           0
         \end{array}\right) + E \left(\begin{array}{c}
           I_{\alpha}\\
           I_{\gamma}
         \end{array}\right) + G \text{\tmop{with}} E \in M_2 (\mathbbm{C}), G
         \in K^2, \]
      where \ $M$ is defined in equation (\ref{eq:M}) below.
      
      \item Moreover, $E$ is symmetric, that is we have $a_X = b_M$ in
      equations (\ref{eq:aX}) and (\ref{eq:bM}) below.
    \end{itemizedot}
    \item Assume that $\Phi \in K$ and conditions (1),(2) hold true with
    $I_{\alpha}$ and $I_{\gamma}$ are dependant, that is, we have an Ostrowski
    relation $I_{\gamma} - \theta I_{\alpha} \in K$. Let's we write (2) into
    the form $\Psi_{\gamma} - eI_{\alpha} \in K$. Then, $\tmop{VE}_{2,
    \alpha}^{\gamma}$ is virtually Abelian iff $M + \theta X$ can be computed
    into a polynomial form
    \[ M + \theta X = \frac{e}{2} I^2 + aI + g \text{\tmop{with}} a \in
       \mathbbm{C}, g \in K. \]
  \end{enumeratenumeric}
\end{theorem}

\begin{proof}
  (1) Again, let us assume that $F / K = \tmop{PV} (\tmop{VE}_{2,
  \alpha}^{\gamma}) / K$ is virtually Abelian. $F$ contains the first level
  integrals $\Psi_{\nu} = \int \Phi I'_{\nu}$ for $\nu \in \{\gamma ;
  \alpha\}$. It also contains the second level integral $\Psi_{2 \alpha} =
  \int \Psi_{\alpha}' I_{\alpha}$ and we get an Ostrowski relation between
  $\Psi_{\alpha}$ and $I_{\alpha}$ thanks to Lemma \ref{lem:fund}.
  
  (2) Since the second level integral $\Psi_{\alpha, \gamma} \assign \int
  \Psi'_{\gamma} I_{\alpha} + \Psi'_{\alpha} I_{\gamma}$ also belongs to $F$,
  and since $F / K$ is virtually Abelian, from Theorem \ref{thliouville}, the
  cocycles,
  \[ \mathcal{C}_{\alpha, \gamma} (\sigma) = \sigma (\Psi_{\alpha, \gamma}) -
     \Psi_{\alpha, \gamma} = c_{\alpha} (\sigma) \Psi_{\gamma} + c_{\gamma}
     (\sigma) \Psi_{\alpha} + l (\sigma) \in T (F_1 / K) = K [I_{\gamma} ;
     I_{\alpha}], \]
  From point (1), we already know that $\Psi_{\alpha} \in T (F_1 / K)$.
  Therefore, $\Psi_{\gamma} \in T (F_1 / K)$ and we get an Ostrowski relation
  between the three integrals $\Psi_{\gamma}$, $I_{\gamma}$ and $I_{\alpha}$.
  
  (3) According to point (1) and Remark \ref{rem:psi}, we can choose a fixed
  representative of $\Phi$ such that $\Psi_{\alpha} \in K$. In other words we
  can assume that $d_{\alpha} = 0$ in (1). Let's compute $\Psi_{\gamma,
  \alpha}$ thanks to the two previous Ostrowski relations modulo integral of
  first level and elements of $T (F_1 / K)$. We get
  \begin{equation}
    \Psi_{\alpha, \gamma} \assign \Psi_{\gamma} I_{\alpha} + \Psi_{\alpha}
    I_{\gamma} - M \text{\tmop{with}} M \assign M_{\alpha, \gamma} = \int
    \Psi_{\gamma} I'_{\alpha} + \Psi_{\alpha} I'_{\gamma} . \label{eq:M}
    (\tmop{eq} : M)
  \end{equation}
  Since $\Psi_{\alpha} \in K$ and $\Psi_{\gamma}$ can be written
  $\Psi_{\gamma} = d_{\gamma} I_{\gamma} + dI_{\alpha} + f$ with $f \in K$,
  the expression $\Psi_{\gamma} I_{\alpha} + \Psi_{\alpha} I_{\gamma} \in T
  (F_1 / K)$ and $M$ is a second level integral belonging to $F$. Moreover,
  \begin{eqnarray*}
    M & = & \int d_{\gamma} I_{\gamma} I'_{\alpha} + dI_{\alpha} I'_{\alpha} +
    (fI_{\alpha}' + \Psi_{\alpha} I_{\gamma}')\\
    M & = & d_{\gamma}  \int I_{\gamma}' I_{\alpha} + \frac{d}{2} I_{\alpha}^2
    + \int fI_{\alpha}' + \Psi_{\alpha} I_{\gamma}' \in F.
  \end{eqnarray*}
  If we set $J_1 \assign \int fI_{\alpha}' + \Psi_{\alpha} I_{\gamma}'$; this
  a first level integral over $K$ and the last relation tells us that \
  \[ d_{\gamma} \int I_{\gamma}' I_{\alpha} + J_1 \in F. \]
  But $F / K$ virtually Abelian implies that $F (J_1) / K$ is also virtually
  Abelian. Hence applying Lemma \ref{lem:fund} again, we get an Ostrowski
  relation between $I_{\gamma}$ and $I_{\alpha}$ if $d_{\gamma} \neq 0$.
\end{proof}

Before proving point (4), let us assume that $\Phi$ is algebraic and
conditions (1),(2) and (3) of the theorem hold true. We have seen that these
conditions can be restated into the following simpler form :
\begin{equation}
  \left\{ \begin{array}{l}
    \Phi \in K, \Psi_{\alpha} \in K\\
    \Psi_{\gamma} = d_{\gamma} I_{\gamma} + dI_{\alpha} + f,\\
    d_{\gamma} \neq 0 \hspace{1em} \Rightarrow \hspace{1em} \exists (\theta,
    \kappa) \in \mathbbm{C}^{\ast} \times K|I_{\gamma} = \theta I_{\alpha} +
    \kappa
  \end{array} \right.
\end{equation}
\label{eq:neccond}Let us set $X \assign \int \Psi_{\alpha} I'_{\alpha}$ and $M
\assign \int \Psi_{\gamma} I'_{\alpha} + \Psi_{\alpha} I'_{\gamma}$ as in
(\ref{eq:M}). $X$ is a first level integral. Here we are going to show in
addition that the second level integrals $M$ and $\Psi_{\gamma, \alpha}$ can
also be computed in closed form. Precisely we shall prove that any such
integral coincides with a polynomial in $I_{\alpha}, I_{\gamma}$ with
coefficients in $K$ plus a first level integral.

\begin{proof}
  of point (4). Since $\Psi_{2 \alpha} = \Psi_{\alpha} I_{\alpha} - X$, and
  $X$ is a first level integral and $\Psi_{2 \alpha}$ can be computed in
  closed form. Now, from point (3) \ we get two possibilities. If $d_{\gamma}
  = 0$, then $M = \frac{d}{2} I_{\alpha}^2 + J_1$. If $d_{\gamma} \neq 0$,
  then
  \[ M = \frac{d_{\gamma} \theta}{2} I_{\alpha}^2 + d_{\gamma}  \int \kappa
     I_{\alpha}' + \frac{d}{2} I_{\alpha}^2 + J_1 = \frac{d_{\gamma} \theta +
     d}{2} I_{\alpha}^2 + J_2 \text{\tmop{with}} J_2 \assign J_1 + d_{\gamma} 
     \int \kappa I_{\alpha}' . \]
  Hence, $M$ can be computed in closed form since $J_2$ is a first level
  integral. For $\Psi_{\gamma, \alpha}$ this follows from (\ref{eq:M}). As a
  consequence, according to Theorem \ref{thliouville}, and Section
  \ref{sec:hierarchy}, $\tmop{VE}_{2, \alpha}^{\gamma}$ is virtually Abelian
  iff $\Phi_{\gamma, 2 \alpha}$ can be computed in closed form.
\end{proof}

\paragraph{Computation of the cocycle $\mathcal{C}_{\gamma, 2 \alpha} (\sigma)
= \sigma (\Psi_{\gamma, 2 \alpha}) - \Psi_{\gamma, 2 \alpha}$ when
(\ref{eq:neccond}) hold true}

By substituting the previous integrals by their expression in closed form into
the formula $\mathcal{C}_{\gamma, 2 \alpha} (\sigma) = 2 c_{\alpha}
\Psi_{\alpha, \gamma} + 2 c_{\gamma} \Psi_{2 \alpha} + c_{\alpha}^2
\Psi_{\gamma} + 2 c_{\alpha} c_{\gamma} \Psi_{\alpha} + l (\sigma)$, we get
\begin{eqnarray}
  \mathcal{C}_{\gamma, 2 \alpha} (\sigma) & = & 2 c_{\alpha}  [\Psi_{\gamma}
  I_{\alpha} + \Psi_{\alpha} I_{\gamma} - M]  \label{eq:cocycle1}\\
  &  & + 2 c_{\gamma}  [\Psi_{\alpha} I_{\alpha} - X] \nonumber\\
  &  & + c_{\alpha}^2  [\Psi_{\gamma}] + 2 c_{\alpha} c_{\gamma} 
  [\Psi_{\alpha}] + l (\sigma) \nonumber
\end{eqnarray}
According to Theorem \ref{thliouville}, $G$ is virtually Abelian implies that
\[ \forall \sigma \in G^{\circ}, \mathcal{C} (\sigma) \in T (F_1 / K) = K
   [I_{\alpha} ; I_{\gamma}] . \]
Therefore, if we compute $\mathcal{C}_{\gamma, 2 \alpha} (\sigma)$ modulo $T
(F_1 / K) = K [I_{\alpha} ; I_{\gamma}]$, we find a new necessary condition
for Abelianity:
\begin{equation}
  G \text{v. \tmop{Ab}} \Rightarrow \forall \sigma \in G^{\circ}, c_{\gamma} X
  + c_{\alpha} M \in K [I_{\alpha} ; I_{\gamma}] . \label{eq:neccond2}
\end{equation}
Two cases therefore happen

{\tmstrong{First Case}} : If $I_{\gamma}$ and $I_{\alpha}$ are independent
over $K$, then (\ref{eq:neccond2}) is equivalent to having both $X$ and $M$ in
$K [I_{\alpha} ; I_{\gamma}]$. Since $X$ is a first level integral this
translates to having an Ostrowski relation of the form
\begin{equation}
  X = a_X I_{\alpha} + b_X I_{\gamma} + g_X \text{\tmop{with}} a_X, b_X \in
  \mathbbm{C}, g_X \in K. \label{eq:aX}
\end{equation}
Since the independence of $I_{\gamma}$ and $I_{\alpha}$ imposes to having
$d_{\gamma} = 0$ that is $\Psi_{\gamma} = dI_{\alpha} + f$, and $M =
\frac{d}{2} I_{\alpha}^2 + J_1$. Therefore, $M \in K [I_{\gamma} ;
I_{\alpha}]$ iff we get an Ostrowski relation of the form
\begin{equation}
  J_1 = a_M I_{\alpha} + b_M I_{\gamma} + g_M \Leftrightarrow M = \frac{d}{2}
  I_{\alpha}^2 + a_M I_{\alpha} + b_M I_{\gamma} + g_M . \label{eq:bM}
\end{equation}
Now, if we substitute the actual expressions in closed polynomial form of
$\Psi_{\gamma}, X$ and $M$ into (\ref{eq:cocycle1}), and expand the result as
a polynomial in $\{I_{\alpha} ; I_{\gamma} \}$ and $\{c_{\alpha} ; c_{\gamma}
\}$ with coefficients in $K$, this gives the following formula for the
cocycle:
\begin{eqnarray*}
  \mathcal{C}_{\gamma, 2 \alpha} (\sigma) & = & d [c_{\alpha} I_{\alpha}^2 +
  c_{\alpha}^2 I_{\alpha}] + l_{\gamma, 2 \alpha} (\sigma)\\
  &  & + [2 c_{\alpha} (f - a_M) + 2 c_{\gamma} (\Psi_{\alpha} - a_X)]
  I_{\alpha}\\
  &  & + [2 c_{\alpha} (\Psi_{\alpha} - b_M) + 2 c_{\gamma} ( - b_X)]
  I_{\gamma} \\
  &  & + c^2_{\alpha} f + 2 c_{\alpha} c_{\gamma} \Psi_{\alpha} - 2
  c_{\alpha} g_M - 2 c_{\gamma} g_X .
\end{eqnarray*}
Again, as in point (4) of Theorem \ref{prop:EXtest2}, this cocycle is of
degree two, and by setting
\[ A \assign \left(\begin{array}{cc}
     2 (f - a_M) & 2 (\Psi_{\alpha} - b_M)\\
     2 (\Psi_{\alpha} - a_X) & - 2 b_X
   \end{array}\right), \tilde{A} \assign \left(\begin{array}{cc}
     f & \Psi_{\alpha}\\
     \Psi_{\alpha} & 0
   \end{array}\right), G \assign \left(\begin{array}{c}
     g_M\\
     g_X
   \end{array}\right), I \assign \left(\begin{array}{c}
     I_{\alpha}\\
     I_{\gamma}
   \end{array}\right), \]
we get a formula
\[ \mathcal{C}_{\gamma, 2 \alpha} (\sigma) = \Delta ( \frac{d}{3}
   I_{\alpha}^3) + C (\sigma)^T AI + C (\sigma)^T  \tilde{A} C (\sigma) - 2 C
   (\sigma)^T G + l (\sigma) . \]
But $A - 2 \tilde{A}$ is a constant matrix, and $A$ is symmetric iff $a_X =
b_M$. So, \ $\mathcal{C}_{\gamma, 2 \alpha}$ is a coboundary iff this latter
condition is satisfied. This proves the criteria given in point (5).

{\tmstrong{Second case}} : Here, we assume that $I_{\gamma}$ and $I_{\alpha}$
are dependant that is $I_{\gamma} = \theta I_{\alpha} + \kappa$ and
$c_{\gamma} = \theta c_{\alpha}$. For simplicity we will write $I_{\alpha} =
I$ and $c_{\alpha} = c \Rightarrow c_{\gamma} = \theta c$. The Ostrowski
relation for $\Psi_{\gamma}$ will be written $\Psi_{\gamma} = eI + f$ with $e
\assign \theta d_{\gamma} + d$. As a consequence the formula for $M$ is now $M
= \frac{e}{2} I^2 + J_2$ (independently of the possible vanishing of
$d_{\gamma}$ the important number with that respect is now $e$). Now
(\ref{eq:neccond2}) is equivalent to having
\[ M + \theta X = \frac{e}{2} I^2 + J_2 + \theta X \in K [I_{\alpha} ;
   I_{\gamma}] = K [I] \Leftrightarrow J_2 + \theta X \in K [I] . \]
Since, $J_2 + \theta X$ is a first level integral this equivalent to having an
expansion in closed polynomial form :
\[ J_2 + \theta X = aI + g \Leftrightarrow M + \theta X = \frac{e}{2} I^2 + aI
   + g \text{\tmop{with}} a \in \mathbbm{C}, g \in K. \]
Again, if we substitute the actual expressions of $\Psi_{\gamma}$ and $M +
\theta X$ in closed polynomial form in (\ref{eq:cocycle1}) with $I_{\gamma} =
\theta I + \kappa$ and $c_{\gamma} = \theta c$ and expand the result as a
polynomial in $I$ and $c$, this gives the following formula for the cocycle
\begin{eqnarray*}
  \mathcal{C}_{\gamma, 2 \alpha} (\sigma) & = & e [cI^2 + c^2 I] + l_{\gamma,
  2 \alpha} (\sigma)\\
  &  & + 2 c (f + 2 \theta \Psi_{\alpha} - a) I + c^2  (f + 2 \theta
  \Psi_{\alpha}) + 2 c (\kappa \Psi_{\alpha} - g) .
\end{eqnarray*}
There is no condition on the size, and we recognise a coboundary
\[ \mathcal{C}_{\gamma, 2 \alpha} (\sigma) = \Delta [ \frac{e}{3} I^3 + (f + 2
   \theta \Psi_{\alpha} - a) I^2 + 2 (\kappa \Psi_{\alpha} - g) I] + l
   (\sigma) . \]
and we do not get more constrain in this case, and point (6) follows.

\section{Effective test for $\tmop{VE}_{2, \alpha}^{\gamma}$}

Here according to Section \ref{sec:hierarchy}, $K =\mathbbm{C} (z) [\omega,
y_1 ; x_1]$. We are therefore led to apply Proposition \ref{prop:test1} with
\[ \Phi = \int \omega y_1 x^2_1 . \]
\subsection{Testing the algebraicity of $\Phi$ when $|k| \geqslant 3$
}\label{sec:testphi}

According to Lemma \ref{lem:jordan}, when $|k| \geqslant 3$ and
$G_{\gamma}^{\circ} = G_{\alpha}^{\circ} = G_a$, we can write $y_1$ and $x_1$
into the form
\[ y_1 = z^{a_{\gamma}}  (z - 1)^{b_{\gamma}} J_{\gamma} (z) \text{\tmop{and}}
   x_1 = z^{a_{\alpha}}  (z - 1)^{b_{\alpha}} J_{\alpha} (z) . \]
Since $\omega = z^{- ( \frac{3}{2} + \frac{1}{2 k})}  (z - 1)^{\frac{-
5}{4}}$, and $\Phi' = \omega y_1 x^2_1$ we get,
\[ \Phi' = z^{E_0}  (z - 1)^{E_1} J_{\gamma} J^2_{\alpha}, \]
with $E_0 = a_{\gamma} + 2 a_{\alpha} - 3 / 2 - 1 / 2 k$, and $E_1 =
b_{\gamma} + 2 b_{\alpha} - 5 / 4$. These values can be explicitly computed
thanks to Table \ref{tabjacobi}, and lead us to the following new table giving
the explicit expression of $z^{E_0}  (z - 1)^{E_1}$ in the expression $\Phi' =
z^{E_0}  (z - 1)^{E_1} J_{\gamma} J^2_{\alpha}$.

\begin{table}[h]
  \begin{tabular}{|c|c|c|c|c|}
    \hline
    $\gamma \backslash \alpha$ & 1 & 2 & 3 & 4\\
    \hline
    1 & $z^{1 / k}  (z - 1)^{- 1 / 2} $ & $z^{1 / k}  (z - 1)^{1 / 2}$ & $z^{-
    1 / k}  (z - 1)^{- 1 / 2}$ & $z^{- 1 / k}  (z - 1)^{1 / 2}$\\
    \hline
    2 & $z^{1 / k}$ alg & $z^{1 / k}  (z - 1)$ alg & $z^{- 1 / k}$ alg & $z^{-
    1 / k}  (z - 1)$ alg\\
    \hline
    3 & $(z - 1)^{- 1 / 2}$ alg & $(z - 1)^{1 / 2}$ alg & $z^{- 2 / k}  (z -
    1)^{- 1 / 2}$ & $z^{- 2 / k}  (z - 1)^{1 / 2}$\\
    \hline
    4 & 1 alg & $(z - 1)$ alg & $z^{- 2 / k}$ alg & $z^{- 2 / k}  (z - 1)$
    alg\\
    \hline
  \end{tabular}
  \caption{\label{tab:phialg}}
\end{table}

In this table the notation {\guillemotleft} alg {\guillemotright}, means that
necessarily, $\Phi$ is algebraic, and we count 10 cases over the 16
possibilities where this happens! In fact, this happens when at least one of
the exponents $E_0$ or $E_1$ belong to $\mathbbm{Z}$. This can be seen by
direct computation, either by a consequence of Remark \ref{rem:fuchs}.

Now,for each of the 6 entries of the previous table, where $\Phi$ can be
transcendental, $E_0 \in \{\pm 1 / k ; - 2 / k\}$ and $E_1 \in \{\pm 1 / 2\}$,
and we may write
\[ \Phi' = \Omega P \text{\tmop{with}} \Omega \assign z^{E_0}  (1 - z)^{E_1}
   \text{\tmop{and}} P \assign J_{\gamma} J_{\alpha}^2 . \]
As a consequence, $\Phi$ is an integral of the type described in the previous
section, so according to Proposition \ref{prop:testphi}, the algebraicity of
$\Phi$ reduces to the vanishing of $\mu (P) = \mu (J_{\gamma} J_{\alpha}^2)$.
There are two ways of testing the vanishing of this number. One with the
coefficients of $P$ if they are explicitly known. And the other with the
evaluation of the definite integral between 0 and 1.

About this second method it applies sometime efficiently in our context.
Indeed, for any Jacobi polynomial $J (z) = J^{(\alpha, \beta)}$, we can
associate a kernel $\Omega = \Omega_J = z^{\beta}  (1 - z)^{\alpha}$, whose
explicit formula is given thanks to Table \ref{tabjacobi}. Moreover we know
that $J$ is orthogonal for the scalar product $< P ; Q > = \int_0^1 PQ
\Omega_J$, to any polynomial whose degree is smaller than $\deg (J)$.

For examples : thanks to Tables \ref{tabjacobi} and \ref{tab:phialg}, we have

In case $(\gamma, \alpha) = (1 ; 1)$ then, $\Omega = \Omega_{\gamma} =
\Omega_{\alpha}$ \ therefore,
\[ \Phi = \int J_{\gamma} J_{\alpha}^2 \Omega_{\gamma} \Rightarrow \mu = \mu
   (J_{\gamma} J_{\alpha}^2) = < J_{\gamma} ; J_{\alpha}^2 >_{\gamma} . \]
Hence, $2 \deg (J_{\alpha}) < \deg (J_{\gamma}) \Rightarrow \mu (P) = 0$ and
the corresponding $\Phi$ is algebraic. This of course give new cases when
$\Phi$ is algebraic. Unfortunately, we have got no converse and $\Phi$ could
be algebraic outside of these cases.

In case $(\gamma, \alpha) = (1 ; 2)$ then $\Omega = \Omega_{\gamma}  (1 - z)
= \Omega_{\alpha}$ therefore,
\[ \Phi = \int J_{\gamma} J_{\alpha}^2  (1 - z) \Omega_{\gamma} \Rightarrow
   \mu = \mu (J_{\gamma} J_{\alpha}^2) = < J_{\gamma} ; J_{\alpha}^2  (1 - z)
   >_{\gamma} . \]
Hence, $2 \deg (J_{\alpha}) + 1 < \deg (J_{\gamma}) \Rightarrow \mu (P) = 0$
and the corresponding $\Phi$ is algebraic.

For the remaining four cases we have unfortunately no comparison of the
expression of $\mu$ with one of the two scalar products. So in general we are
not able to give any condition for the algebraicity ! Nevertheless, for
applications in specific examples one is therefore reduced just to check the
relation given by point 3 of Proposition \ref{prop:testphi}.

\subsection{Getting obstruction when $\Phi$ can be transcendental in the 6
possible cases of Table \ref{tab:phialg}}

Our main result is going to be the following

\begin{proposition}
  \label{prop:R1ve2}When $|k| \geqslant 3$ and $\Phi$ is transcendental, then
  in the six cases of Table \ref{tab:phialg}, $\tmop{VE}_{2, \alpha}^{\gamma}$
  is not virtually Abelian excepted maybe for $|k| = 3$ when
  $(\lambda_{\gamma} ; \lambda_{\alpha}) \in \tmop{Case} 3 \times
  \{\tmop{Case} 3 ; \tmop{Case} 4\}$.
\end{proposition}

\begin{proof}
  Let's assume that $\tmop{VE}_{2, \alpha}^{\gamma}$ is virtually Abelian and
  look to the restrictions imposed by this condition thanks to Proposition
  \ref{prop:test1} and its consequences. According to Lemma \ref{lem:jordan},
  the action of the monodromy around $z = 0$, is given by characters thanks to
  the following formulae
  \begin{equation}
    \mathcal{M}_0 (I'_{\gamma}) = \exp (\mathi 2 \pi ( - 2 a_{\gamma}))
    I_{\gamma}', \mathcal{M}_0 (\Phi') = \exp (\mathi 2 \pi E_0) \Phi' .
    \label{eq:char}
  \end{equation}
  Since $\Phi$ is transcendental, the two constants $d_{\gamma}$ and
  $d_{\alpha}$ in Proposition \ref{prop:test1}.1 both are non-zero. Therefore,
  according to the later, we get an Ostrowski between $I_{\gamma}$ and
  $I_{\alpha}$. As a consequence, thanks to Lemma \ref{lem:char}, the
  Ostrowski relation between $I_{\gamma}$ and $I_{\alpha}$ implies that
  \[ \exp (\mathi 2 \pi ( - 2 a_{\gamma})) = \exp (\mathi 2 \pi ( - 2
     a_{\alpha})) \Leftrightarrow 2 a_{\gamma} - 2 a_{\alpha} \in \mathbbm{Z}.
  \]
  But, by writing $a = \frac{1}{2} + \frac{\varepsilon}{2 k}$, with
  $\varepsilon = 1$ in Cases 1 or 2, $\varepsilon = - 1$ in Cases 3 or 4, we
  get that $2 a_{\gamma} - 2 a_{\alpha} \in \mathbbm{Z}$ iff,
  $\varepsilon_{\gamma} = \varepsilon_{\alpha}$. In other words we proved that
  an Ostrowski relation between $I_{\gamma}$ and $I_{\alpha}$ implies that
  \begin{equation}
    a_{\gamma} = a_{\alpha} \Leftrightarrow (\gamma ; \alpha) \in
    \{\tmop{Cases} 1 \tmop{or} 2\}^2 \cup \{\tmop{Cases} 3 \tmop{or} 4\}^2 .
    \label{eq:osI}
  \end{equation}
  This in fact eliminates two possibilities in Table \ref{tab:phialg}.
  
  Now, the Ostrowski relation between $\Phi$ and $I_{\gamma}$ implies
  similarly that $E_0 + 2 a_{\gamma} \in \mathbbm{Z}$. But $E_0 = a_{\gamma} +
  2 a_{\alpha} - 3 / 2 - 1 / 2 k$ and $a_{\gamma} = a_{\alpha}$, therefore,
  \[ E_0 + 2 a_{\gamma} = 5 a_{\gamma} - \frac{3}{2} - \frac{1}{2 k} = 1 +
     \frac{5 \varepsilon_{\gamma} - 1}{2 k} . \]
  For $\varepsilon = 1$, (Cases 1 or 2), $\frac{5 \varepsilon - 1}{2 k} =
  \frac{2}{k} \nin \mathbbm{Z}$, Hence $\tmop{VE}_{2, \alpha}^{\gamma}$ is not
  virtually Abelian.
  
  For $\varepsilon = - 1$, (Cases 3 or 4), $\frac{5 \varepsilon - 1}{2 k} =
  \frac{- 3}{k} \in \mathbbm{Z} \Leftrightarrow |k| = 3$. This proves that in
  general $\tmop{VE}_{2, \alpha}^{\gamma}$ is not virtually Abelian except
  maybe in the exceptional cases mentioned in the proposition. 
\end{proof}

\subsection{Getting Obstruction when $\Phi$ is algebraic }

Here, according to the second point of Theorem \ref{prop:ve2test2}, we have to
test an Ostrowski relation between $\Psi_{\alpha}$ and $I_{\alpha}$. Again, if
$\Psi_{\alpha}$ is algebraic such a relation hold and we get nothing new. As a
consequence our first task is going to be the study of this problem.

\subsubsection{Testing the algebraicity of $\Psi_{\alpha}$ and
$\Psi_{\gamma}$}

Here, according to Lemma \ref{lem:jordan},
\[ \Psi_{\alpha}' = \Phi I_{\alpha}' = \frac{\Phi}{x_1^2} = \frac{\Phi}{z^{2
   a_{\alpha}}  (z - 1)^{2 b_{\alpha}} J_{\alpha}^2 (z)}, \]
and similar formula with $\Psi_{\gamma}' = \Phi I_{\gamma}' = \Phi / y_1^2$.

According to Remark \ref{rem:psi}, we must compute the two $\Psi'$ with the
same $\Phi$ up to an additive constant. The precise forms of those expressions
are respectively given for $\Psi_{\alpha}'$ resp for $\Psi_{\gamma}'$ in the
following two tables

\begin{table}[h]
\[
     \begin{array}{|c|c|c|c|c|}
       \hline
       \gamma \backslash \alpha & 1 & 2 & 3 & 4\\
       \hline
       1 & \frac{R}{J^2_{\alpha}}  & \frac{R}{J^2_{\alpha}} &
       \frac{R}{J^2_{\alpha}} & \frac{R}{J^2_{\alpha}}\\
       \hline
       2 & \frac{Q}{(z - 1)^{1 / 2} J_{\alpha}^2} & \frac{Q}{(z - 1)^{3 / 2}
       J_{\alpha}^2} & \frac{Q}{(z - 1)^{1 / 2} J_{\alpha}^2} & \frac{Q}{(z -
       1)^{3 / 2} J_{\alpha}^2}\\
       \hline
       3 & \frac{Q}{z^{1 + 1 / k} J_{\alpha}^2} & \frac{Q}{z^{1 + 1 / k}
       J_{\alpha}^2} & \frac{R}{z^{1 / k} J^2_{\alpha}} & \frac{R}{z^{1 / k}
       J^2_{\alpha}}\\
       \hline
       4 & \frac{Q}{z^{1 / k}  (z - 1)^{1 / 2} J_{\alpha}^2} & \frac{Q}{z^{1 /
       k}  (z - 1)^{3 / 2} J_{\alpha}^2} & \frac{Q}{z^{1 / k}  (z - 1)^{1 / 2}
       J_{\alpha}^2} & \frac{Q}{z^{1 / k}  (z - 1)^{3 / 2} J_{\alpha}^2}\\
       \hline
     \end{array}
\]
     \caption{\label{tab:psialpha}}
   \end{table}

\begin{table}[h]
\[
     \begin{array}{|c|c|c|c|c|}
       \hline
       \gamma \backslash \alpha & 1 & 2 & 3 & 4\\
       \hline
       1 & \frac{R}{J^2_{\gamma}} & \frac{(z - 1) R}{J^2_{\gamma}} &
       \frac{R}{z^{2 / k} J^2_{\gamma}} & \frac{(z - 1) R}{z^{2 / k}
       J^2_{\gamma}}\\
       \hline
       2 & \frac{Q}{(z - 1)^{3 / 2} J_{\gamma}^2} & \frac{Q}{(z - 1)^{3 / 2}
       J_{\gamma}^2} & \frac{Q}{z^{2 / k} (z - 1)^{3 / 2} J_{\gamma}^2} &
       \frac{Q}{z^{2 / k} (z - 1)^{3 / 2} J_{\gamma}^2}\\
       \hline
       3 & \frac{Q}{z^{1 - 1 / k} J_{\gamma}^2} & \frac{(z - 1) Q}{z^{1 - 1 /
       k} J_{\gamma}^2} & \frac{R}{z^{1 / k} J^2_{\gamma}} & \frac{(z - 1)
       R}{z^{1 / k} J^2_{\gamma}}\\
       \hline
       4 & \frac{Q}{z^{- 1 / k} (z - 1)^{3 / 2} J_{\gamma}^2} & \frac{Q}{z^{-
       1 / k} (z - 1)^{3 / 2} J_{\gamma}^2} & \frac{Q}{z^{1 / k} (z - 1)^{3 /
       2} J_{\gamma}^2} & \frac{Q}{z^{1 / k} (z - 1)^{3 / 2} J_{\gamma}^2}\\
       \hline
     \end{array}
\]
     \caption{\label{tab:psigamma}}
   \end{table}
Let's explain briefly how those tables were computed. When we are in the ten
cases of Table \ref{tab:phialg}, where $\Phi$ is always algebraic, then $\Phi'
= z^{E_0}  (z - 1)^{E_1} P$, with at least one integral exponent. Assume for
instance that $E_1 \in \mathbbm{N}$, then we choose a primitive of the form
$\Phi = z^{E_0 + 1} Q$ with $Q \in \mathbbm{C} [z]$.

Now, when we are in the six remaining cases when the two exponents are not
integers,
\[ \Phi = \int \Omega P = \int z^{E_0}  (z - 1)^{E_1} P. \]
According to Theorem \ref{th:redint} points 1 and 2, $\Phi$ is algebraic iff
it can be computed in closed form
\[ \Phi = \Mho R = z^{E_0 + 1}  (z - 1)^{E_1 + 1} R \text{\tmop{with}} R \in
   \mathbbm{C} [z] . \]
Surprisingly, in those six cases we get the following implication
\[ (\Phi \tmop{alg}) \Rightarrow (\Psi_{\alpha} \tmop{and} \Psi_{\gamma}
   \tmop{alg}) . \]
Indeed in those six cases, $E_1 + 1 \in \{1 / 2 ; 3 / 2\}$ For $\nu \in
\{\alpha ; \gamma\}$,
\[ \Psi'_{\nu} = \frac{z^{E_0 + 1}  (z - 1)^{E_1 + 1} R}{z^{2 a_{\nu}}  (z -
   1)^{2 b_{\nu}} J_{\nu}^2 (z)} . \]
Since $2 b_{\nu}$ also belongs to $\{1 / 2 ; 3 / 2\}$, the exponent of
$\Psi'_{\nu}$ at $z = 1$ is an integer $\neq - 1$. As a consequence,
$\Psi_{\nu}$ is fixed by $\mathcal{M}_1$ and we can conclude thanks to Remark
\ref{rem:fuchs}.

As a consequence, when $\Phi$ is algebraic, we see form the above tables that
$\Psi_{\alpha}$ and $\Psi_{\gamma}$ are simultaneously algebraic very often.
In fact in 10 cases over the 16. Since it is much more difficult to prove the
transcendence than to show the algebraicity of an integral, we therefore
cannot say something general in the six remaining cases, when for example
$\lambda_{\gamma} \in \tmop{Case} 4$.

\subsubsection{Getting obstruction when $\Psi_{\alpha}$ or $\Psi_{\gamma}$ is
transcendental}

Here our main result is going to be the following

\begin{proposition}
  \label{prop:R2ve2}If $\Phi$ is algebraic, and $\Psi_{\alpha}$ is
  transcendental, then $\tmop{VE}_{2, \alpha}^{\gamma}$ is not virtually
  Abelian except maybe for the two cases $(\lambda_{\gamma} ;
  \lambda_{\alpha}) \in \tmop{Case} 4 \times \{\tmop{Case} 1 \tmop{or} 2\}$.
  
  Similarly, if $\Psi_{\gamma}$ is transcendental, then $\tmop{VE}_{2,
  \alpha}^{\gamma}$ is not virtually Abelian excepted maybe for the two cases
  $(\lambda_{\gamma} ; \lambda_{\alpha}) \in \tmop{Case} 4 \times
  \{\tmop{Case} 1 \tmop{or} 2\}$ or when $(\lambda_{\gamma} ;
  \lambda_{\alpha}) \in \tmop{Case} 2 \times \{\tmop{Case} 3 \tmop{or} 4\}$ if
  $|k| = 3$.
\end{proposition}

\begin{proof}
  From now we assume that $\tmop{VE}_{2, \alpha}^{\gamma}$ is virtually
  Abelian. According to Theorem \ref{prop:ve2test2}, we must get two Ostrowski
  relations: one between $\Psi_{\alpha}$ and $I_{\alpha}$, and another one
  between $\Psi_{\gamma}$, $I_{\gamma}$ and $I_{\alpha}$. Let's assume that we
  get such a relation between $\Psi_{\gamma}$ $I_{\gamma}$ and $I_{\alpha}$.
  The arguments for $\Psi_{\alpha}$ being simpler. According to Lemma
  \ref{lem:char}, if $\Psi_{\gamma}$ is transcendental, we must get the
  following relation between the exponents at $z = 0$ :
  \[ e_0 (\Psi_{\gamma}') \in \{e_0 (I'_{\gamma}) ; e_0 (I'_{\alpha})\}
     \tmop{mod} \mathbbm{Z}. \]
  Now a direct comparison, of Table \ref{tab:psigamma}, with Table
  \ref{tab:I}, in the six cases where $\Psi_{\gamma}$ can be transcendental,
  gives that $e_0 (\Psi_{\gamma}') \in \{e_0 (I'_{\gamma}) ; e_0
  (I'_{\alpha})\} \tmop{mod} \mathbbm{Z}$, when $(\lambda_{\gamma} ;
  \lambda_{\alpha}) \in \tmop{Case} 4 \times \{\tmop{Case} 1 \tmop{or} 2\}$ or
  when $(\lambda_{\gamma} ; \lambda_{\alpha}) \in \tmop{Case} 2 \times
  \{\tmop{Case} 3 \tmop{or} 4\}$, if $|k| = 3$. Moreover, when
  $(\lambda_{\gamma} ; \lambda_{\alpha}) \in \tmop{Case} 4 \times
  \{\tmop{Case} 1 \tmop{or} 2\}$, we get the same conclusion for the exponents
  at $z = 1$. This is the reason why, these two cases cannot be a-priory
  refined.
\end{proof}

\subsubsection{Some results when $\Phi, \Psi_{\alpha}$ and $\Psi_{\gamma}$ are
algebraic}

As we said before those cases happen very often in 10 cases over the the 16.
This the right moment to apply Theorem \ref{prop:ve2test2}, with $d_{\alpha} =
d_{\gamma} = e = 0$. Here, $X = \int \Psi_{\alpha} I_{\alpha}'$ and $M = \int
\Psi_{\alpha} I_{\gamma}' + \Psi_{\gamma} I_{\alpha}'$ are first level
integrals and we first have to see if they can be computed in polynomial form.
That is, if they satisfied Ostrowski relations with $I_{\alpha}$ and
$I_{\gamma}$...

Precisely, according to the theorem we will have two cases
\begin{itemizedot}
  \item When $I_{\alpha}$ and $I_{\gamma}$ are independent then we must check
  the following simple form of the relation given in point 5 of the theorem
  \begin{equation}
    \left(\begin{array}{c}
      M\\
      X
    \end{array}\right) = E \left(\begin{array}{c}
      I_{\alpha}\\
      I_{\gamma}
    \end{array}\right) + G, \label{eq:MXind}
  \end{equation}
  where $E$ is symmetric.
  
  \item When \ $I_{\alpha}$ and $I_{\gamma}$ are dependant, that is
  $I_{\gamma} - \theta I_{\alpha} \in K$, then we must check a relation of the
  form
  \begin{equation}
    M + \theta X = aI + g. \label{eq:MXdep}
  \end{equation}
\end{itemizedot}
Now, according to equation (\ref{eq:osI}), we have a necessary condition for
the independence of the two integrals $I_{\alpha}$ and $I_{\gamma}$, which we
resume in the following new table

\begin{table}[h]
  \begin{tabular}{|c|c|c|c|c|}
    \hline
    $\gamma \backslash \alpha$ & Case 1 & Case 2 & Case 3 & Case 4\\
    \hline
    Case 1 & ? & ? & Ind & Ind\\
    \hline
    Case 2 & ? & ? & Ind & Ind\\
    \hline
    Case 3 & Ind & Ind & ? & ?\\
    \hline
    Case 4 & Ind & Ind & ? & ?\\
    \hline
  \end{tabular}
  \caption{\label{tabIdep}}
\end{table}

So we see again, that in half of the cases the integrals are independent, and
for the remaining cases we obviously do not know.

We have made explicit computations of the integrals $X$ and $M$, they are
first level integrals so they will obey the rule given by Remark
\ref{rem:fuchs}. Nevertheless, they are integrals of the form
\[ \int \frac{P \Omega}{J_{\alpha} J_{\gamma}^2} \text{\tmop{or}} \int
   \frac{P \Omega}{J_{\gamma} J_{\alpha}^2} . \]
As a consequence, they do not enter into the context of Theorem
\ref{th:redint}.

\subsection{Experimental considerations for $\tmop{VE}_{2,
\alpha}^{\gamma}$.}\label{sec:exp}

Here we are going to join the information given by Propositions
\ref{prop:R1ve2} and \ref{prop:R2ve2}, the previous tables and some
experiments given by computers. This will give a picture of the behaviour of
the Galois group of $\tmop{VE}_{2, \alpha}^{\gamma}$. For simplicity, we will
assume that
\[ |k| \geqslant 5. \]
Indeed, the Propositions above showed, that for $|k| = 3$ some exceptional
behaviour occur. Experiments also show that for $|k| = 4$, we also get new
exceptions. This is quite normal since these two cases correspond
geometrically to the situation where the hyper-elliptic curve parametrised by
$t \mapsto (\varphi (t), \dot{\varphi} (t))$ is in fact an elliptic curve.

\paragraph{Step 1: when $\Phi$ is transcendental}

According to Section \ref{sec:testphi}, we found that $\Phi$ can be
transcendental in 6 over the 16 cases, with half of the possibilities when
$(\gamma, \alpha) \in \{(1, 1) ; (1, 2)\}$. It seems that experiments are
showing that away from the predicted cases, $\Phi$ is transcendental. As a
consequence, thanks to Proposition \ref{prop:R1ve2}, the probability to get
obstruction when $\Phi$ is transcendental is
\[ p_{\Phi} = 5 / 16. \]

\paragraph{Step 2: when $\Phi$ is algebraic}

Here for convenience, we present a new table resuming the case when $\Phi$ is
algebraic

\begin{table}[h]
  \begin{tabular}{|c|c|c|c|c|}
    \hline
    $\gamma \backslash \alpha$ & Case 1 & Case 2 & Case 3 & Case 4\\
    \hline
    Case 1 & $A$ & $A'$ &  & \\
    \hline
    Case 2 & $B$ & $B'$ & $C$ & $C'$\\
    \hline
    Case 3 & $D$ & $D'$ &  & \\
    \hline
    Case 4 & $E$ & $E'$ & $F$ & $F'$\\
    \hline
  \end{tabular}
  \caption{}
\end{table}

The empty cases correspond to the four cases where $\Phi$ is transcendental.
Again, for each letter $X$, we will denote by $p_X$ the probability to get
obstruction in the context given by the corresponding letter. Here, the first
thing to check is the possible transcendence of either $\Psi_{\alpha}$ or
$\Psi_{\gamma}$ in order to apply Proposition \ref{prop:R2ve2}. If nothing
subsequent occurs from this test, then we check one of the two relations
(\ref{eq:MXind}) or (\ref{eq:MXdep}) depending of the possible dependence of
the integrals $I_{\alpha}$, $I_{\gamma}$ in order to apply Theorem
\ref{prop:ve2test2}.

\subparagraph{In $A$ and $A'$:}According to Section \ref{sec:testphi}, $\Phi$
is algebraic when $2 \deg (J_{\alpha}) < \deg (J_{\gamma})$ in $A$ and, when
$2 \deg (J_{\alpha}) + 1 < \deg (J_{\gamma})$ in $A'$. In both cases,
according to Tables \ref{tab:psialpha} and \ref{tab:psigamma} $\Psi_{\alpha}$
and $\Psi_{\gamma}$ are algebraic. Moreover, relation (\ref{eq:MXind}) hold
for some symmetric matrix $E$. Observe that here, we do not need to check the
dependence of the integrals $I$, since if this happens then we would get
(\ref{eq:MXind})$\Rightarrow$(\ref{eq:MXdep}). As a consequence, there is no
obstruction and
\[ p_A = p_{A'} = 0. \]
For the remaining letters, $\Phi$ is always algebraic.

\subparagraph{In $B$ and $B'$:}According to Tables \ref{tab:psialpha} and
\ref{tab:psigamma}, $\Psi_{\alpha}$ and $\Psi_{\gamma}$ are algebraic.
Moreover, experiments give that $M$ and $X$ are algebraic. As a consequence,
relation (\ref{eq:MXind}) is trivially satisfied with $E = 0$. Hence,
\[ p_B = p_{B'} = 0. \]
\subparagraph{In $C$ and $C'$:}According to Table \ref{tab:psialpha}
$\Psi_{\alpha}$ is algebraic. Nevertheless, experience shows that
$\Psi_{\gamma}$ is transcendental, hence according to Proposition
\ref{prop:R2ve2}, $\tmop{VE}_{2, \alpha}^{\gamma}$ \ is not virtually Abelian
and,
\[ p_C = p_{C'} = 1 / 16. \]
\subparagraph{In $D$ and $D'$:}According to Tables \ref{tab:psialpha} and
\ref{tab:psigamma}, $\Psi_{\alpha}$ and $\Psi_{\gamma}$ are algebraic. Here,
thanks to Table \ref{tabIdep}, the two integrals $I_{\alpha}$ and $I_{\gamma}$
are independent. In both cases, experiments show that $M$ is algebraic and $X$
is transcendental. In $D$, we get an Ostrowski relation of the form $X =
aI_{\alpha} + g$. Therefore, in (\ref{eq:MXind}) the matrix
\[ E = \left(\begin{array}{cc}
     0 & 0\\
     a & 0
   \end{array}\right), \text{\tmop{with}} a \neq 0, \]
is not symmetric. Therefore, there is obstruction. In $D'$ the situation is
quite similar excepted that there is no Ostrowski relation between $X,
I_{\alpha}$ and $I_{\gamma}$, so (\ref{eq:MXind}) is not satisfied. Hence,
\[ p_D = p_{D'} = 1 / 16. \]
\subparagraph{$\tmop{In} E$ and $E'$:} Here, we cannot apply directly
Proposition \ref{prop:R2ve2}. And in fact experiments show that we get two
Ostrowski relations, one between $\Psi_{\alpha}$ and $I_{\alpha}$ and the
other between $\Psi_{\gamma}$ and $I_{\gamma}$. Moreover, $\Psi_{\alpha}$ and
$\Psi_{\gamma}$ are transcendental. But experiments also show that there is no
possible Ostrowski relation between $M$, $I_{\alpha}$ and $I_{\gamma}$.
Therefore, none of the equations (\ref{eq:MXind}) or (\ref{eq:MXdep}) can be
satisfied. So $\tmop{VE}_{2, \alpha}$ is not virtually Abelian and,
\[ p_E = p_{E'} = 1 / 16. \]
\subparagraph{In $F$ and $F'$:}Experiments show that: in $F$ both
$\Psi_{\alpha}$ and $\Psi_{\gamma}$ are algebraic iff $2 \deg (J_{\alpha}) >
\deg (J_{\gamma})$. Similarly, in $F'$, both $\Psi_{\alpha}$ and
$\Psi_{\gamma}$ are algebraic iff $2 \deg (J_{\alpha}) + 1 > \deg
(J_{\gamma})$. Moreover, if these conditions on the degrees of the Jacobi
polynomials are satisfied, then $M$ and $X$ are algebraic. Hence, we get half
obstruction in each case and,
\[ p_F = p_{F'} = 1 / 32. \]
As a consequence, the probability to get obstruction when $\Phi$ is algebraic
is
\[ p_{\tmop{alg}} = p_A + \cdots + p_{F'} = 7 / 16. \]
Hence, the total probability to get obstruction is
\[ p_T = p_{\Phi} + p_{\tmop{alg}} = 5 / 16 + 7 / 16 = 3 / 4. \]
To our point of view, there are two significant conclusions that can be
derived, from this study : First, that there is still a lot a obstruction at
the level of the second variational equation. Indeed, it seems that there is a
quite big distance between solvable Galois groups and virtually Abelian ones.
Secondly, although it is comparatively much more complicated to test, the most
important obstruction to the virtual Abelianity of $\tmop{VE}_{2,
\alpha}^{\gamma}$ happens when $\Phi$ is algebraic.

\section{Considerations about $\tmop{EX}_{2, \alpha, \beta}^{\gamma}$}

Here we follow the same strategy. But now the main technical difficulty comes
in distinguishing the 64 possibilities for the cases satisfied by $\gamma,
\beta, \alpha$. We will therefore have to deal with spatial tables of 64
entries ! As a consequence, we will not give definitive results. Nevertheless,
Theorem \ref{prop:EXtest2}, can be used to deal with the complete study of
specific potentials.

\subsection{Counting the cases when $\Phi = \int \omega u_1 y_1 x_1$ is
algebraic}

Since the most check able obstructions are going to happen when $\Phi$ is
transcendental, at first glance, we count the possible numbers of such
occurrences. Here, by using similar arguments as in Table \ref{tab:phialg}, we
are going to count 44, possibilities where $\Phi$ is algebraic. Therefore, we
will be left with at most 20 cases where $\Phi$ can be transcendental !

Here,
\[ u_1 = z^{a_{\gamma}}  (z - 1)^{b_{\gamma}} J_{\gamma} ; y_1 = z^{a_{\beta}}
   (z - 1)^{b_{\beta}} J_{\beta} ; x_1 = z^{a_{\alpha}}  (z - 1)^{b_{\alpha}}
   J_{\alpha} \Rightarrow \Phi' = z^{E_0}  (z - 1)^{E_1} P (z), \]
With $P (z) = J_{\gamma} J_{\beta} J_{\alpha}$ and,
\[ E_0 = a_{\gamma} + a_{\beta} + a_{\alpha} - \frac{3}{2} - \frac{1}{2 k} ;
   E_1 = b_{\gamma} + b_{\beta} + b_{\alpha} - \frac{5}{4} . \]
Again, as in Section \ref{sec:tobe}, $\nocomma \Phi$ is going to be algebraic
when at least one of the two exponents $E_0$ or $E_1$ is an integer distinct
from -1.

\subsubsection{Values of $E_0$}

If we write $a = \frac{1}{2} + \frac{\varepsilon}{2 k}$ with $\varepsilon = 1$
in Cases 1 or 2 and $\varepsilon = - 1$ in Cases 3 or 4, we get
\[ E_0 = \frac{\varepsilon_{\gamma} + \varepsilon_{\beta} +
   \varepsilon_{\alpha} - 1}{2 k} . \]
Then $E_0 \in \{1 / k ; 0 ; - 1 / k ; - 2 / k\}$, when $\tmop{Card} (i|
\varepsilon_i = 1) \in \{3 ; 2 ; 1 ; 0\}$. As a consequence,
\[ E_0 \in \mathbbm{Z} \Leftrightarrow E_0 = 0 \Leftrightarrow \tmop{Card} \{i
   \in \{\gamma, \beta, \alpha\}| \varepsilon_i = 1\} = 2. \]
This happens in the $3 \times 8 = 24$ possibilities listed in the following
table

\begin{table}[h]
  \begin{tabular}{|c|c|c|c|}
    \hline
    $E_0 \in \mathbbm{Z}$ & $\gamma$ & $\beta$ & $\alpha$\\
    \hline
    cases $L_1$ & 1 or 2 & 1 or 2 & 3 or 4\\
    \hline
    cases $L_2$ & 1 or 2 & 3 or 4 & 1 or 2\\
    \hline
    cases $L_3$ & 3 or 4 & 1 or 2 & 1 or 2\\
    \hline
  \end{tabular}
  \caption{\label{tab:E0int}}
\end{table}

\subsubsection{Values of $E_1$}

Since $b = 1 / 4$ in Cases 1 or 3 and $b = 3 / 4$ in Cases 2 or 4, we get
\[ \left\{ \begin{array}{lllll}
     3 \tmop{cases} 1 \tmop{or} 3 & \Rightarrow & E_1 = - 1 / 2 &  & \\
     2 \tmop{cases} 1 \tmop{or} 3 & \Rightarrow & E_1 = 0 & \in \mathbbm{Z} &
     24 \tmop{possibilities}\\
     1 \tmop{cases} 1 \tmop{or} 3 & \Rightarrow & E_1 = 1 / 2 &  & \\
     0 \tmop{cases} 1 \tmop{or} 3 & \Rightarrow & E_1 = 1 & \in \mathbbm{Z} &
     8 \tmop{possibilities}
   \end{array} \right. \]
As a consequence $E_1 \in \mathbbm{Z}$ in 32 possibilities.

\subsubsection{Counting when $E_0$ and $E_1$ both are integers}

This happens when either $(E_0 ; E_1) = (0 ; 1)$ either $(E_0 ; E_1) = (0 ;
0)$.

\paragraph{When $(E_0 ; E_1) = (0 ; 1)$}Then, $E_1 = 1$ implies that $(\gamma
; \beta ; \alpha) \in \{\tmop{Cases} 2 \tmop{or} 4\}^3$ and the intersection
with the table for $E_0 = 0$, gives 3 possibilities which are the cyclic
permutations of
\[ (\gamma ; \beta ; \alpha) \in (2 ; 2 ; 4) . \]
\paragraph{When $(E_0 ; E_1) = (0 ; 0)$}Then $E_1 = 0$ implies the following
possibilities

\begin{table}[h]
  \begin{tabular}{|c|c|c|c|}
    \hline
    $E_1 = 0$ & $\gamma$ & $\beta$ & $\alpha$\\
    \hline
    cases $L'_1$ & 1 or 3 & 1 or 3 & 2 or 4\\
    \hline
    cases $L'_2$ & 1 or 3 & 2 or 4 & 1 or 3\\
    \hline
    cases $L'_3$ & 2 or 4 & 1 or 3 & 1 or 3\\
    \hline
  \end{tabular}
  \caption{}
\end{table}

As a consequence, if we compute $L_i \cap L'_j$ we get $9 = 3 \times 3$
distinct cases. For example, $L_1 \cap L'_1$ implies $(\gamma ; \beta ;
\alpha) \in (1 ; 1 ; 4)$.

Therefore, the two exponents both are integers in $3 + 9 = 12$ cases.

\subsubsection{Conclusion}

Since we get 24 possibilities for $E_0 \in \mathbbm{Z}$, 32 possibilities for
$E_1 \in \mathbbm{Z}$ and 12 for both integral exponents. $\Phi$ is going to
be algebraic in $44 = 24 + 32 - 12$ cases. As a consequence, in the spatial
table of 64 entries for $(\gamma ; \beta ; \alpha)$, there are at most 20
possibilities where $\Phi$ can be transcendental.

\subsection{Counting when all the $\Psi_{\mu}$ are algebraic in the 44 case
where $\Phi$ is algebraic}

This is also made in order to find the cases where there is no obstruction.

In fact thanks to the consideration above we count 16 cases over 64 where
everybody is certainly algebraic.

\subsection{Getting obstruction with the assumption that $\Phi$ is
transcendental }

Here our main result is going to be the following one which is very similar in
its statement and proof to Proposition \ref{prop:R1ve2}.

\begin{proposition}
  \label{prop:R1EX}For $|k| \geqslant 3$, in the 20 possible cases when $\Phi$
  can be transcendental, $\tmop{EX}_{2, \alpha, \beta}^{\gamma}$ is not
  virtually Abelian excepted maybe for $|k| = 3$ and $(\gamma ; \beta ;
  \alpha) \in \{\tmop{Cases} 3 \tmop{or} 4\}^3$.
\end{proposition}

\begin{proof}
  From now we assume that $\tmop{EX}_{2, \alpha, \beta}^{\gamma}$ is virtually
  Abelian. According to Proposition \ref{prop:test1}, we get three Ostrowski
  relations $\Phi + d_{\gamma} I_{\gamma}$, $\Phi + d_{\beta} I_{\beta}$ and
  $\Phi + d_{\alpha} I_{\alpha}$ are algebraic over $\mathbbm{C} (z)$. Since
  $\Phi$ is transcendental, the three constants $d_i$ are non-zero. As a
  consequence we get three Ostrowski relations between any two $I_i$ and
  $I_j$, for $i \neq j$. But we have seen in the proof of Proposition
  \ref{prop:R1ve2}, that such a relation implies that $a_i = a_j$ (see
  equation(\ref{eq:osI})). And we can therefore deduce that
  \[ (\gamma ; \beta ; \alpha) \in \{\tmop{Cases} 1 \tmop{or} 2\}^3 \cup
     \{\tmop{Cases} 3 \tmop{or} 4\}^3 . \]
  Now, the Ostrowski relation between $\Phi$ and $I_{\gamma}$ implies
  similarly that $E_0 + 2 a_{\gamma} \in \mathbbm{Z}$. But $E_0 = a_{\gamma} +
  a_{\beta} + a_{\alpha} - 3 / 2 - 1 / 2 k$ and $a_{\gamma} = a_{\beta} =
  a_{\alpha}$, therefore,
  \[ E_0 + 2 a_{\gamma} = 5 a_{\gamma} - \frac{3}{2} - \frac{1}{2 k} = 1 +
     \frac{5 \varepsilon_{\gamma} - 1}{2 k} . \]
  When $\varepsilon = 1$, (Cases 1 or 2), $\frac{5 \varepsilon - 1}{2 k} =
  \frac{2}{k} \nin \mathbbm{Z}$, Hence $\tmop{EX}_{2, \alpha, \beta}^{\gamma}$
  is not virtually Abelian.
  
  When $\varepsilon = - 1$, (Cases 3 or 4), $\frac{5 \varepsilon - 1}{2 k} =
  \frac{- 3}{k} \in \mathbbm{Z} \Leftrightarrow |k| = 3$. This proves that in
  general $\tmop{EX}_{2, \alpha, \beta}^{\gamma}$ is not virtually Abelian
  excepted maybe in the exceptional cases mentioned in the proposition.

\end{proof}

\subsection{Getting obstruction in the 44 cases when $\Phi$ is algebraic}

In order to be able to exploit the Ostrowski relations given by Theorem
\ref{prop:EXtest2}, in these 44 cases, we first have to investigate the
possible algebraicity of the $\Psi_{\mu}$, when $\mu \in \{\gamma, \beta,
\alpha\}$.

\subsubsection{About the algebraicity of the integrals $\Psi_{\mu}$ for $\mu
\in \{\gamma, \beta, \alpha\}$}

Here $\Psi'_{\mu} = \Phi I_{\mu}' = z^{E_0}  (z - 1)^{E_1} Q (z) I'_{\mu}$.
Let's denote $N_0^{\mu}$ and $N_1^{\mu}$ the respective exponents of
$\Psi'_{\mu}$ at $z = 0$ and $z = 1$, respectively. Up to addition of a
positive integer we get
\[ N_0^{\mu} = E_0 - 2 a_{\mu} ; N_1^{\mu} = E_1 - 2 b_{\mu} . \]
Direct computation gives
\[ N_0^{\mu} \in \{- 1, \pm 1 / k - 1, - 2 / k - 1\}; N_1^{\mu} \in \{- 3 /
   2, - 1, - 1 / 2, 0, 1 / 2\}. \]

\subsubsection{Getting obstruction when at least one $\Psi_{\mu}$ is
transcendental}

\begin{proposition}
  \label{prop:R2EX}When $\Phi$ is algebraic in the 44 \ cases mentioned above
  we get :
  \begin{enumeratenumeric}
    \item Let us assume that $|k| \geqslant 4$. If at least one of the
    $\Psi_{\mu}$ is transcendental, then $\tmop{EX}_{2, \alpha,
    \beta}^{\gamma}$ is not virtually Abelian excepted maybe in the 12 cases
    where $E_0$ and $E_1$ both are integers.
    
    \item For$ |k| = 3$, we get the same conclusion if we assume that at least
    two of the three integrals $\Psi_{\mu}$ are transcendental.
  \end{enumeratenumeric}
\end{proposition}

\begin{proof}
  Now we assume that $\tmop{EX}_{2, \alpha, \beta}^{\gamma}$ is virtually
  Abelian. By symmetry let's assume that $\Psi_{\gamma}$ is transcendental.
  According to Theorem \ref{prop:EXtest2}, we must have a non trivial
  Ostrowski relation
  \[ \Psi_{\gamma} + d_{\gamma} I_{\gamma} + d_{\beta} I_{\beta} + d_{\alpha}
     I_{\alpha} \in K. \]
  The monodromy $\mathcal{M}_1$ acts by characters on the derivatives of these
  four integrals. According to Lemma \ref{lem:char}.2, the character of
  $\Psi_{\gamma}$ must be equal to the character of one of the $I_{\mu}$. But,
  as in the proof of Proposition \ref{prop:R2ve2}, this condition is
  equivalent to having
  \[ N^{\gamma}_1 - \frac{1}{2} \in \mathbbm{Z} \Leftrightarrow E_1 - 2
     b_{\gamma} - \frac{1}{2} \in \mathbbm{Z} \Leftrightarrow E_1 \in
     \mathbbm{Z}. \]
  Now let's do the same job with $\mathcal{M}_0$. As in the proof of
  Proposition \ref{prop:R2ve2}, the Ostrowski relation imposes that at least
  one of the three following numbers is an integer:
  \[ N^{\gamma}_0 + 2 a_{\gamma} = E_0 ; N^{\gamma}_0 + 2 a_{\beta} = E_0 - 2
     a_{\gamma} + 2 a_{\beta} ; N^{\gamma}_0 + 2 a_{\alpha} = E_0 - 2
     a_{\gamma} + 2 a_{\alpha} . \]
  Let us set $\Delta_{\beta, \alpha} \assign E_0 - 2 a_{\beta} + 2
  a_{\alpha}$. With this notation, we just have seen that $\Psi_{\gamma}$
  transcendental implies that
  \[ E_0 \in \mathbbm{Z} \text{\tmop{or}} \Delta_{\gamma, \beta} \in
     \mathbbm{Z} \text{\tmop{or}} \Delta_{\gamma, \alpha} \in \mathbbm{Z}. \]
  If $E_0 \in \mathbbm{Z}$, both exponent at $z = 0$ and $z = 1$ are integers,
  we are in the 12 cases over the 44 where $\Phi$ is algebraic and these
  arguments do not give any obstruction to the virtual Abelianity of
  $\tmop{EX}_{2, \alpha, \beta}^{\gamma}$.
  
  Now let's assume that $E_0 \nin \mathbbm{Z}$. We are led to find
  obstructions from the conditions $\Delta \in \mathbbm{Z}$. But, by writing
  $a = 1 / 2 + \varepsilon / 2 k$, we get
  \[ \Delta_{\beta, \alpha} = \frac{\varepsilon_{\gamma} - \varepsilon_{\beta}
     + 3 \varepsilon_{\alpha} - 1}{2 k} . \]
  Its values depends on the eight possibilities given by $\varepsilon_{\mu} =
  \pm 1$. They are listed in the following table
  
  \begin{table}[h]
    \begin{tabular}{|c|c|c|c|c|}
      \hline
      & $\varepsilon_{\gamma}$ & $\varepsilon_{\beta}$ &
      $\varepsilon_{\alpha}$ & $\Delta_{\beta, \alpha}$\\
      \hline
      $L_0$ & + & + & + & $1 / k$\\
      \hline
      $L_4$ & - & - & - & $- 2 / k$\\
      \hline
      $L_5$ & - & - & + & $1 / k$\\
      \hline
      $L_6$ & - & + & - & $- 3 / k$\\
      \hline
      $L_7$ & + & - & - & $- 1 / k$\\
      \hline
    \end{tabular}
    \caption{}
  \end{table}
  
  Here, we did not give the value of $\Delta_{\beta, \alpha}$ for the three
  lines $L_1, L_2, L_3$ of Table \ref{tab:E0int}, because they correspond to
  $E_0 \in \mathbbm{Z}$.
  
  From this table, we get that $\Delta \nin \mathbbm{Z}$ except when $|k| =
  3$, in the case of line $L_6$. This prove the first point of the
  proposition.
  
  For the second point, let's assume again that $E_0 \nin \mathbbm{Z}$ and
  $\Psi_{\gamma}$ is transcendental. We must have $\Delta_{\gamma, \beta} \in
  \mathbbm{Z}$ or $\Delta_{\gamma, \alpha} \in \mathbbm{Z}$. But according to
  line $L_6$ we have the implications
  \[ \Delta_{\gamma, \beta} \in \mathbbm{Z} \Rightarrow (\gamma, \beta,
     \alpha) = (+, -, -) \Rightarrow \Delta_{\gamma, \alpha} \in \mathbbm{Z}.
  \]
  Therefore, if for example $\Psi_{\beta}$ is transcendental, we would get,
  according to the table
  \[ \Delta_{\beta, \alpha} = - 1 / k \nin \mathbbm{Z}, \hspace{2em}
     \Delta_{\beta, \gamma} = 1 / k \nin \mathbbm{Z}. \]
  So, the Ostrowski relation will not be satisfied for $\Psi_{\beta}$. This
  prove the claim.
\end{proof}


\begin{thebibliography}{1}
  \bibitem[1]{weil12}A.~Aparicio~Monforte and J.-A. Weil. {\newblock}A
  reduction method for higher order variational equations of Hamiltonian
  systems. {\newblock}In \tmtextit{Symmetries and related topics in
  differential and difference equations}, volume 549 of \tmtextit{Contemp.
  Math.}, pages 1--15. Amer. Math. Soc., Providence, RI, 2011.
  
  \bibitem[2]{combot12}Thierry Combot. {\newblock}Non-integrability of the
  equal mass; n-body problem with non-zero angular momentum.
  {\newblock}\tmtextit{Celestial Mechanics and Dynamical Astronomy}, pages
  1--22.
  
  \bibitem[3]{duval08}Guillaume Duval and Andrzej~J. Maciejewski.
  {\newblock}Jordan obstruction to the integrability of Hamiltonian systems
  with homogeneous potentials. {\newblock}\tmtextit{Annales de l'Institut
  Fourier}, 59(7):2839--2890, 2009.
  
  \bibitem[4]{duval12}Guillaume Duval and Andrzej~J. Maciejewski.
  {\newblock}Integrability of Homogeneous potential of degree $k = \pm 2$. An
  application of higher variational equations. {\newblock}\tmtextit{submited},
  2012.
  
  \bibitem[5]{moralesramis01}Juan~J. Morales-Ruiz and Jean~Pierre Ramis.
  {\newblock}A note on the non-integrability of some Hamiltonian systems with
  a homogeneous potential. {\newblock}\tmtextit{Methods Appl. Anal.},
  8(1):113--120, 2001.
  
  \bibitem[6]{poole}E.~G.~C. Poole. {\newblock}\tmtextit{Introduction to the
  theory of linear differential equations}. {\newblock}Dover Publications
  Inc., New York, 1960.
\end{thebibliography}
\end{document}